\documentclass[twoside,a4paper,11pt,limits]{amsart}
\usepackage{graphicx}
\usepackage{dsfont}

\usepackage{srcltx}
\usepackage{amssymb}
\usepackage{amsmath}
\usepackage{amssymb}
\usepackage{amsthm}
\usepackage{bbm}
\usepackage{hyperref}
\hypersetup{colorlinks=true,linkcolor=blue,citecolor=red}
\allowdisplaybreaks

\usepackage[hyperref,comment,enumitem]{paper_diening}

\newtheorem{theorem}{Theorem}[section]
\newtheorem{lemma}[theorem]{Lemma}

\newtheorem{remark}[theorem]{Remark}

\newtheorem{definition}[theorem]{Definition}

\numberwithin{equation}{section}

\newcommand{\dint}{\dashint}

\newcommand{\Mcal}{{\mathcal{M}}}

\newcommand{\pk}[1]{{#1}}

\newcommand{\dt}{\ensuremath{\,{\rm d} t}}
\newcommand{\ds}{\ensuremath{\,{\rm d} s}}

\newcommand{\dx}{\ensuremath{\,{\rm d} x}}
\newcommand{\dy}{\ensuremath{\,{\rm d} y}}
\newcommand{\dz}{\ensuremath{\,{\rm d} z}}

\newcommand{\seb}[1]{{#1}}

\begin{document}

\title[Higher integrability for parabolic systems]{On global $L^q$ estimates for systems with $p$-growth in rough domains}
\thanks{M.~Bul\'{\i}\v{c}ek's and S.~Schwarzacher's work is supported by the project LL1202  financed by the Ministry of Education, Youth and Sports,
Czech Republic and by the University Centre for Mathematical Modelling, Applied Analysis and Computational Mathematics (Math~MAC).
M.~Bul\'{\i}\v{c}ek is a  member of the Ne\v{c}as Center for Mathematical Modeling}

\author[M.~Bul\'{\i}\v{c}ek]{Miroslav Bul\'{\i}\v{c}ek} 
\address{Mathematical Institute, Faculty of Mathematics and Physics, Charles University,
Sokolovsk\'{a} 83, 186 75 Prague, Czech Republic}
\email{mbul8060@karlin.mff.cuni.cz}

\author[S.~Byun]{Sun-Sig Byun}
\address{Department of Mathematical Sciences and Research Institute of Mathematics, Seoul National University,
Seoul 08826, Korea}
\email{byun@snu.ac.kr}

\author[P.~Kaplick\'{y}]{Petr Kaplick\'{y}}
\address{Department of Mathematical Analysis, Faculty of Mathematics and Physics,  Charles University,
Sokolovsk\'{a} 83, 186 75 Prague, Czech Republic}
\email{kaplicky@karlin.mff.cuni.cz}

\author[J.~Oh]{Jehan Oh}
\address{Fakult\"{a}t f\"{u}r Mathematik, Universit\"{a}t Bielefeld,
Postfach 100131, D-33501 Bielefeld, Germany}
\email{joh@math.uni-bielefeld.de}

\author[S.~Schwarzacher]{Sebastian Schwarzacher}
\address{
Institute of applied mathematics, University of Bonn, Ende
nicher Allee 60, 53115 Bonn, Germany \&
Department of Analysis, Faculty of Mathematics and Physics,  Charles University,
Sokolovsk\'{a} 83, 186 75 Prague, Czech Republic
}
\email{schwarz@karlin.mff.cuni.cz}

\begin{abstract}
We study regularity results for nonlinear parabolic systems of $p$-Laplacian type with
inhomogeneous boundary and initial data, with $p\in
(\frac{2n}{n+2},\infty)$. We show bounds on the gradient of
solutions in the Lebesgue-spaces with arbitrary large integrability
exponents and natural dependences on the right hand side and the
boundary data. In particular, we provide a new proof of the global
non-linear Calder\'{o}n--Zygmund theory for
such systems. This extends the
global result of~\cite{Bog14} to very rough domains and more general
boundary values. Our method makes use of direct estimates on the
solution minus its boundary values and hence is considerably shorter than the available higher
integrability results. Technically interesting is the fact that our
parabolic estimates have no scaling deficit with respect to the
leading order term. Moreover, in the singular case,
$p\in(\frac{2n}{n+2},2]$, any scaling deficit can be omitted.
\end{abstract}


\maketitle

\section{Introduction}
\label{S1}

\noindent We are interested in higher integrability results for
solutions of the following nonlinear problem: for given data, i.e.,
an $n$-dimensional domain $\Omega
\subset \mathbb{R}^n$ with $n\ge 2$, a length of the time interest
$T>0$,  a source term $f:(0,T)\times \Omega \to \mathbb{R}^{n\times
N}$ with $N \in \mathbb{N}$, an initial condition $u_0:\Omega\to
\setR^N$, a given lateral boundary
condition $g_0:(0,T)\times\partial\Omega\to \setR^N$  and a mapping
$A:\mathbb{R} \times \mathbb{R}^n \times \mathbb{R}^{n\times N}\to
\mathbb{R}^{n\times N}$, we aim to
find a function $u:[0,T)\times\Omega \to \mathbb{R}^N$ satisfying
\begin{align}
\label{eq:sysA}
\left\lbrace \ \begin{aligned}
  \partial_tu-\divergence A(t,x, \nabla u) &= -\divergence
  f &&\textrm{ in }(0,T)\times \Omega,\\
  u&=g_0&&\textrm{ on }(0,T)\times \partial \Omega,\\
  u(0,\cdot)&=u_0(\cdot) &&\textrm{ in } \Omega.
  \end{aligned} \right.
\end{align}
Owing to a significant number of problems originating in various applications, it is natural to require that $A$ is a Carath\'{e}odory mapping, satisfying the $p$-coercivity, the $(p-1)$-growth and the (strict) monotonicity conditions.
It means that
\begin{align}
&A(\cdot,\cdot,\eta) \textrm{ is measurable for any fixed } \eta\in \mathbb{R}^{n\times N},\label{Car1}\\
&A(t,x,\cdot) \textrm{ is continuous for almost all } (t,x) \in \mathbb{R} \times \mathbb{R}^n, \label{Car2}
\end{align}
and for some $p\in (1,\infty)$, there exist positive constants $c_1$ and $c_2$ such that for almost all $(t,x) \in \mathbb{R} \times \mathbb{R}^n$ and all $\eta_1,\eta_2 \in \mathbb{R}^{n\times N}$
\begin{align}
\label{coercivityJM}
c_1|\eta_1|^p-c_2 &\le A(t,x,\eta_1)\cdot \eta_1 &&\textrm{$p$-coercivity},\\
|A(t,x,\eta_1)|&\le c_2(1+|\eta_1|)^{p-1}&&\textrm{$(p-1)$-growth},\label{growthJM}\\
0&\le (A(t,x,\eta_1)-A(t,x,\eta_2))\cdot (\eta_1-\eta_2)&&\textrm{monotonicity}.\label{monotoneJM}
\end{align}
If for all $\eta_1\neq \eta_2$ the inequality \eqref{monotoneJM} is strict, then $A$ is said to be strictly monotone.

The leading motive of the paper is to find assumptions on the data,
i.e., $\partial \Omega$, $A$,
$u_0$, $g_0$ and $f$, that would guarantee not only the existence of
a weak solution $u\in L^p(0,T;W^{1,p}(\Omega; \mathbb{R}^N))$ but
also its higher integrability in the sense that for
any $q \in [1,\infty)$ depending on the data, we have
\begin{equation}
u\in L^{pq}(0,T;W^{1,pq}(\Omega; \mathbb{R}^N)). \label{task}
\end{equation}

Immediately, there appear two key difficulties when one
intends to establish the validity
of \eqref{task}: the nonlinear and vectorial structure of the
operator $A$ and the roughness of the domain $\Omega$, which is
supposed to be just an open bounded set. Both these difficulties
will finally lead to some restrictions that we shall describe below.

To introduce the novelty of our results, we give a short history on
the derivation of \eqref{task} for the model elliptic case. Let
$g\in W^{1,p}(\Omega;\setR^N)$ and $\Omega$ be an open
bounded domain such that
\begin{align} \label{p-lap-eq}
\left\lbrace \ \begin{aligned} -\divergence(\abs{\nabla
u}^{p-2}\nabla u)&=0\text{ in }\Omega
\\
u&=\pk{g}\text{ on }\partial\Omega.
\end{aligned}\right.
\end{align}
It is classical that
\[
u=\argmin_{v\in \pk{g}+ W^{1,p}_0(\Omega)}\frac1p\int_\Omega
\abs{\nabla v}^p\, dx
\]
is the unique solution to the system \eqref{p-lap-eq}. After the
seminal works of Ural'ceva~\cite{Ura68} and Uhlenbeck~\cite{Uhl77},
solutions to systems of \eqref{p-lap-eq} are known to be locally in
$C^{1,\alpha}$ for some $\alpha \in
(0,1)$. In the scalar case, this was extended up to the boundary by
Lieberman~\cite{Lie88}, provided $u_0\in
C^{1,\alpha}(\overline{\Omega})$. The result by Lieberman was the
starting point for many results in the {\em scalar} case. In
particular, higher integrability~\eqref{task} was shown for various generalizations of
the non-linearity, the coefficients,
the right hand sides and the
boundary regularity~\cite{AceM05, Mi1, CM1, KinZho01, ByunOhWang15,
ByunOk16, ByunPala14, ByunWangZhou07, MengPhuc12, Ph1, PalaSoft11}.
\pk{All these results} relay on the so called comparison principle that is
in the heart of the {\em non-linear \Calderon{} Zygmund theory}
first developed by Iwaniec~\cite{Iwa82,Iwa83}, see
also~\cite{CafPer98}. Later, building on the seminal works of
DiBenedetto and Friedman~\cite{DiBFri85,DiB93},
some higher regularity results
could be transferred to the
parabolic setting~\cite{AceMin07, Bog14, Bog15, ByunOk16SIAM,
ByunOkRyu13}.

However, as far as we are concerned,
as of now up to the boundary, $C^{1,\alpha}$ regularity is not
available for systems in case of inhomogeneous boundary
data, as the scalar techniques are
not applicable. The best regularity available for systems (elliptic
and parabolic) are global Lipschitz bounds~\cite{Bog15}, provided
the boundary of the domain $\Omega$ is smooth. These could be used
to prove global higher integrability results for quite general
operators~\cite{Bog14}.

In this paper we provide a new and independent proof of the {\em
global \Calderon{} Zygmund theory for \pk{elliptic and} parabolic systems}. In
contrast to the available higher integrability results of the
kind~\eqref{task}, we show estimates {\em directly} for the
difference $v=u-\pk{g}$. The key observation is, that $v$ satisfies the
following  PDE where the {\em inhomogeneity} of the boundary values
is transferred to a right hand
side: \begin{align*} \left\lbrace \
\begin{aligned} -\divergence(\abs{\nabla v}^{p-2}\nabla
v)&=-\divergence(\abs{\nabla v}^{p-2}\nabla v-\abs{\nabla
u}^{p-2}\nabla u) \text{ in }\Omega
\\
v&=0\text{ on }\partial\Omega.
\end{aligned}
\right.
\end{align*}
As was already pointed out by Uhlenbeck, local regularity for
homogeneous boundary data follows by the mere fact that solutions
can (locally) be reflected. Indeed, local estimates up to the
boundary for homogeneous boundary data but inhomogeneous right hand
side \pk{are in~\cite{BCDKS17}}.

Our proof of \eqref{task}  follows by combining local estimates for
homogeneous boundary values with algebraic properties of the tensor
$A$. The key estimate related to homogeneous boundary values that
may well be of independent interest is an estimate for the gradient
of the parabolic $p$-Laplacian with {\em no scaling deficit} for the
leading order terms (Theorem~\ref{thm:plaphom}).
 It can then be used to show respective estimates with {\em no scaling deficit} on
 the leading order terms for~\eqref{eq:sysA}  (see Theorem~\ref{thm:main}).
 Moreover, in the singular case, $p<2$, any scaling deficit can be omitted (see Remark~\ref{rem:main}).
 This considerably improves the estimates in \cite{Bog14} {\em quantitatively}
 which we consider to be a major scientific contribution of the present work.
 Additionally we extend the theory of the global non-linear \Calderon{} Zygmund theory
 also {\em qualitatively} by allowing more general assumptions on the regularity of the boundary
 and the boundary values in comparison to~\cite{Bog14}. Finally we wish to remark,
 that the assumptions on the non-linearity $A$ in this paper are essentially equivalent to
 the assumptions in~\cite{Bog14}.

\subsection{Assumptions on the boundary}
First, concerning the roughness of $\Omega$, we should be able to
specify at least in which sense the boundary and the initial data,
respectively, are attained. Since $\Omega$
goes beyond the Lipschitz category,
we can hardly define a trace of any Sobolev (or Bochner-Sobolev)
function by using any kind of a trace theorem and therefore we need
to proceed more carefully. Nevertheless, we can be inspired by the
elliptic setting. Indeed, if one considers the problem
$$
\left\lbrace \ \begin{aligned}
-\divergence A(x, \nabla u) & = 0 &&\textrm{in } \Omega,\\
u&=\pk{g}_0 &&\textrm{on } \partial \Omega,
\end{aligned} \right.
$$
then a concept of weak solution ``attaining" the trace $\pk{g}_0$ can be
naturally expressed as follows: to look for $u = \pk{g} + v$ with $v \in
W^{1,p}_0(\Omega; \mathbb{R}^N)$ such that the above equation is
satisfied in the sense of distribution \pk{ and $g$ is fixed such that it attains $g_0$ on $\partial\Omega$ in a suitable way.} Notice that here we
defined for any $p\in
[1,\infty)$,
$$
W^{1,p}_0(\Omega; \mathbb{R}^N):= \overline{\left\{u\in \mathcal{C}^{\infty}_0(\Omega;\mathbb{R}^N) \right\}}^{\|\cdot \|_{1,p}},
$$
where $\|\cdot \|_{1,p}$ is a standard Sobolev norm. Consequently,
we see that the natural assumption on $\pk{g}_0$ is that it must be
extendable into the interior of $\Omega$ such that \pk{its extension $g$ satisfies  $g \in
W^{1,p}(\Omega;\mathbb{R}^N)$}
integrability result, we need to improve also the integrability
of the extension $\pk{g}$ \pk{simultaneously}. In the parabolic setting, we have to face
the very similar problems on the parabolic cylinder $(0,T)\times
\Omega$, \pk{and} we also require that the initial data are attained at
least weakly.


Therefore, for the nonlinear problem \eqref{eq:sysA} having the
nonlinearity $A$ of $(p-1)$-growth, we are directly led to the
following characterization of the boundary data, which will be done
through the extension into the interior of the parabolic cylinder
$(0,T)\times \Omega$.
\begin{definition}\label{DEF:1}
Let $\tilde{p}\in (1,\infty)$ and $q\in [1,\infty)$ be arbitrary
given. We say that the couple
$(u_0,g_0)$ where $u_0:\Omega \to \mathbb{R}^N$ and $g_0:(0,T)\times
\Omega \to \mathbb{R}^N$,  is  an {\bf $\pk{(\tilde{p},q)}$-suitable
couple of boundary and initial data}, if there exists $g\in
W^{1,\tilde{p}'q}((0,T);W^{-1,\tilde{p}'q}(\Omega;\setR^N))\cap
L^{\tilde{p}q}((0,T); W^{1,\tilde{p}q}(\Omega;\setR^N))$ such that
\begin{align*}
\lim_{t\to 0}\int_{\Omega} (g(t,x)-u_0(x)) \cdot \phi(x) \dx &= 0 &&\text{ for all }\phi\in L^\infty(\Omega;\setR^N)\\
\intertext{and} \lim_{\epsilon\to
0}\dashint_{B_\epsilon(x)\cap
\Omega}\abs{g(t,y)-g_0(t,x)}\dy &=0 &&\text{ for almost all }
(t,x)\in(0,T)\times \partial \Omega.
\end{align*}
\end{definition}
Before we demonstrate that the above definition is really natural
for the setting, we consider in the paper, we shall briefly define
all symbols used in the above definition. First, we used the
notation
$$
\dashint_{Z}f(x) \dx:= \frac{1}{|Z|} \int_{Z} f(x) \dx,
$$
where $Z$ is a bounded subset of
$\mathbb{R}^n$. Second, for arbitrary $s\in (1,\infty)$, we also
defined the dual space
$$
W^{-1,s}(\Omega; \mathbb{R}^N):= \left( W^{1,s'}_0(\Omega; \mathbb{R}^N) \right)^*,
$$
where $s'$ is the dual exponent to $s$, i.e., $s':=s/(s-1)$. Finally, for $X$ being a Banach space, we introduced the Bochner or Sobolev-Bochner spaces, respectively, in a usual way. In addition, to simplify the notation, we will also use the following abbreviation in what follows:
$$
\begin{aligned}
X^{\tilde{p},q}&:=W^{1,\tilde{p}'q}((0,T);W^{-1,\tilde{p}'q}(\Omega;\setR^N))\cap L^{\tilde{p}q}((0,T); W^{1,\tilde{p}q}(\Omega;\setR^N)),\\
X^{\tilde{p},q}_0&:=W^{1,\tilde{p}'q}((0,T);W^{-1,\tilde{p}'q}(\Omega;\setR^N))\cap L^{\tilde{p}q}((0,T); W_0^{1,\tilde{p}q}(\Omega;\setR^N)).
\end{aligned}
$$

Next, let us discuss the choice of the suitable couple according to Definition~\ref{DEF:1}. In case, we set $q=1$, then due to the growth and coercivity assumptions on the nonlinear operator $A$, see \eqref{coercivityJM}--\eqref{growthJM}, then $L^p$-compatible data mean nothing else that, for $f\in L^{p'}(0,T; L^{p'}(\Omega; \mathbb{R}^{n\times N}))$, we are able to construct a weak solution $u\in L^p(0,T; W^{1,p}(\Omega; \mathbb{R}^N))$ which attains in the sense of Definition~\ref{DEF:1} boundary and initial data. Then the information that $\partial_t u$ belongs to $L^{p'}(0,T; W^{-1,p'}(\Omega; \mathbb{R}^N))$ can be directly read from \eqref{eq:sysA}. More precisely, we introduce the following definition of a weak solution to \eqref{eq:sysA}--\eqref{monotoneJM}.
\begin{definition}\label{DEF:2}
Let $p \in (\frac{2n}{n+2},\infty)$, $A$ satisfy \eqref{Car1}--\eqref{growthJM}, $f\in L^{p'}(0,T; L^{p'}(\Omega;\mathbb{R}^{n\times N}))$, and let $(u_0,g_0)$ be an $\pk{(p,1)}$-suitable couple of boundary and initial data with an extension $g$. We say that $u$ is a weak solution to \eqref{eq:sysA}, if there exists $v$ such that $u=v+g$ and
\begin{align}
&v\in X^{p,1}_0 \cap \mathcal{C}([0,T]; L^2(\Omega;\mathbb{R}^N)),\label{D1}\\
&\lim_{t\to 0+} \|v(t,\cdot)\|_{L^2(\Omega)} =0,\label{D2}\\
\intertext{and for almost all $t\in (0,T)$ and all $\varphi \in W^{1,p}_0(\Omega;\mathbb{R}^N)$ there holds}
&\langle \pk{v}(t), \varphi\rangle + \int_{\Omega} A(t,x,\nabla \pk{v}(t,x)) \cdot \nabla \varphi(x) \dx = \int_{\Omega}f(t,x)\cdot \nabla \varphi(x) \dx.\label{D3}
\end{align}
\end{definition}
\pk{The duality $\langle v(t), \varphi\rangle$ is understood in $(W^{1,p}(\Omega,\mathbb R^N))^*$.} We will also frequently use the equivalent reformulation of \eqref{D3}, which takes the form
\begin{equation}
\langle v(t), \varphi\rangle + \int_{\Omega} A(t,x,\nabla v(t,x)+\nabla g(t,x)) \cdot \nabla \varphi(x) \dx =-\langle g(t), \varphi\rangle +\int_{\Omega}f(t,x)\cdot \nabla \varphi(x) \dx.\label{D4}
\end{equation}
Finally, to justify the setting of \pk{$(p,q)$-suitable} data, let us
just discuss the case when we want to have $|\nabla u|\in L^{pq}$
for some $q>1$. In view of Definition~\ref{DEF:2}, it is natural to
assume that also $|\nabla g|$ and also $|\nabla v|$ belong to
$L^{pq}$ and for the right hand side one should require $|f|\in
L^{p'q}$. Then due to \eqref{growthJM}, we also have that
$|A(x,\nabla u)|\in L^{p'q}$ and consequently looking to \eqref{D3},
we see that also $\partial_t u$ should belong to $L^{p'q}(0,T;
W^{-1,p'q}(\Omega;\mathbb{R}^N))$. Since $u$ is given as a sum of
$g$ and $v$, it is quite obvious \pk{where} the compatibility
condition on $g$ comes from. In case that the boundary of $\Omega$
is smooth enough, one can trace the precise assumption on
$(u_0,g_0)$ for example by using the sharp theory for the heat
equation together with the trace theorems for Sobolev-Bochner
functions. Since this topic for a rough boundary goes beyond the aim of the paper and in fact is more
related to the function spaces theory than to theory of PDEs, we do
not investigate it in more details and rather formulate all
assumptions in terms of the existence of a suitable extension $g$,
which is included in Definition~\ref{DEF:1}.

\subsection{Assumptions on the non-linearity}

Next assumptions that need to be imposed are related to the boundary $\partial \Omega$ and also to the structure of nonlinearity $A$. In case we are interested only in existence of a weak solution in the sense of Definition~\ref{DEF:2} for an operator $A$ fulfilling \eqref{Car1}--\eqref{monotoneJM}, one can apply the standard Minty method (see \cite{Minty63}) to observe that for any $f\in L^{p'}(0,T; L^{p'}(\Omega; \mathbb{R}^{n\times N}))$ and any suitable $L^p$-couple $(u_0,g_0)$ there exists a function $u\in X^{p,1}$, which is a weak solution to \eqref{eq:sysA} in the sense of Definition~\ref{DEF:2}. Moreover, for fixed extension $g$ of a couple $(u_0,g_0)$ the weak solution $u$ is unique.

However, in case we want to deal with higher integrability results up to the boundary of $(0,T)\times \Omega$, we need to impose more assumptions on $A$ and also on the domain $\Omega$. Concerning the structural assumptions on $A$, it is well documented with counterexamples in various dimensions~\cite{SvYa02} that in case of systems the above assumptions \eqref{Car1}--\eqref{monotoneJM} are not enough to show higher integrability even locally in $(0,T)\times \Omega$. The results in the higher integrability theory for systems are basically known for the model $p$-Laplace case:
\begin{align}
\label{p-lap}
A(x,\eta)=\abs{\eta}^{p-2}\eta,
\end{align}
or \pk{its} extension to radially symmetric case, see \cite{Uhl77} for results in the elliptic setting and also \cite{DiB93} for the parabolic one,
\begin{align}
\label{Uhlenbeck}
\begin{aligned}
A(x,\eta)&=a(x,\abs{\eta})\eta
\end{aligned}
\end{align}
or the most generally, to operators $A$'s behaving asymptotically (when $\abs{\eta}\to \infty$) as \eqref{Uhlenbeck}, see the results \cite{ByunOhWang15,BulDieSch16,BulSch16,CafPer98} for the elliptic setting or \cite{BreStrVer18,ZhaZhe16, Bog14, KuuMin12} for parabolic one. Inspired by these results we will require in what follows that $A$ satisfies the following asymptotically Uhlenbeck condition.
\begin{definition}
A nonlinearity $A(t,x,\eta)$ is said to be {\bf asymptotically Uhlenbeck} if there exists a bounded nonnegative function $\Phi : [0,\infty) \to [0,\infty)$ such that
\begin{equation}\label{AH1}
\lim_{\rho \to \infty} \Phi(\rho) = 0
\end{equation}
and
\begin{equation}\label{AH2}
\abs{A(t,x,\eta) - \abs{\eta}^{p-2}\eta} \le \Phi(\abs{\eta}) \left( 1+\abs{\eta}^{p-1} \right)
\end{equation}
for almost every $(t,x) \in \mathbb{R} \times \mathbb{R}^n$ and all $\eta \in \mathbb{R}^{n\times N}$.
\end{definition}

We remark that if $A(t,x,\eta)$ is asymptotically Uhlenbeck, then
\begin{equation*}
\lim_{\abs{\eta} \to \infty} \frac{A(t,x,\eta) - \abs{\eta}^{p-2}\eta}{\abs{\eta}^{p-1}} = 0
\end{equation*}
uniformly with respect to $(t,x) \in \mathbb{R} \times \mathbb{R}^n$.
Also notice that \eqref{AH1} and \eqref{AH2} directly imply that there exists a constant $c>0$ such that for almost all $(t,x)\in (0,T)\times \Omega$ and all $\eta_1, \eta_2\in \mathbb{R}^{n\times N}$ we have that
\begin{equation}
\label{LC}
|A(t,x,\eta_1)-A(t,x,\eta_2)|\le c\left(1+ (|\eta_1|+|\eta_2|)^{p-2}|\eta_1-\eta_2|\right).
\end{equation}

Finally, we describe our assumption on the domain $\Omega$. Here, we are inspired by the result for Reifenberg flat domains in scalar case, see \cite{ByunOh18, ByunOhWang15, ByunPala14, MengPhuc12, PalaSoft11}. Hence,  we will consider in the paper that the domain $\Omega$ is $(\delta,r_0)$-Reifenberg flat domain. More precisely:

\begin{definition}
\label{def:reif}
We say that $\Omega$ is $(\delta,r_0)$-Reifenberg flat if for each $x \in \partial \Omega$ and for each $r \in (0,r_0]$, there exists a coordinate system $\{y^1, \cdots ,y^n\}$ such that $x=0$ in this coordinate system and that
\begin{equation*}
B_{r}(0)\cap \{ y^n > \delta r \}\subset B_{r}(0) \cap \Omega \subset B_{r}(0) \cap \{y^n > -\delta r \}.
\end{equation*}
\end{definition}

Having collected all necessary assumptions, we can state our main theorem.
\begin{theorem}
\label{thm:main} Let $p\in (\frac{2n}{n+2},\infty)$, $N\in \setN$
and $A$ satisfy \eqref{Car1}--\eqref{monotoneJM} and
\eqref{AH1}--\eqref{AH2}. Then for every $q_0 \in [1,\infty)$, there
exist a small constant $\delta_0 > 0$ depending on $n$, $N$, $p$,
$q_0$ and $A$ such that for any $q \in [1,q_0]$ and $\delta \in
(0,\delta_0]$, we have the following properties: if $\Omega\subset
\setR^n$ is a bounded $(\delta,r_0)$-Reifenberg flat domain for some
$r_0>0$, i.e., it satisfies Definition~\ref{def:reif}, if $f\in
L^{p'q}((0,T; L^{p'q}(\Omega; \mathbb{R}^{n\times N}))$
and if $(u_0,g_0)$ is a suitable
$\pk{(p,q)}$-couple with an extension $g\in X^{p,q}$, then the solution
$u$ to \eqref{eq:sysA} belongs to $X^{p,q}$ and the following
estimate holds
\begin{equation}
\label{Th:ap} \norm{u}_{X^{p,q}} \leq
c\left(1+\norm{g}_{X^{p,q}}+\norm{f}_{L^{p'q}(0,T; L^{p'q} (\Omega;
\mathbb{R}^{n\times
N}))}^\frac{1}{p-1}+\norm{g}_{X^{p,q}}^{\frac{2+p(q-1)}{2q}}
+\norm{f}_{L^{p'}(0,T; L^{p'}(\Omega; \mathbb{R}^{n\times
N}))}^{\frac{2+p(q-1)}{2(p-1)q}} \right),
\end{equation}
with a positive constant $c$ depending only on $n$, $N$, $p$, $q_0$, $r_0$, $\Omega$, $T$ and $A$.
\end{theorem}
\begin{remark}[$p \leq 2$]
\label{rem:main}
Clearly, since in case $p\leq 2$, we have $1+\frac{p}{2}(q-1)\leq q$ and so
\eqref{Th:ap} can be written in form {\em without} a scaling deficit:
\begin{equation}
\label{Th:ple2}
\norm{u}_{X^{p,q}} \leq c\left(1+\norm{g}_{X^{p,q}}+\norm{f}_{L^{p'q}(0,T; L^{p'q}(\Omega; \mathbb{R}^{n\times N}))}^\frac{1}{p-1} \right).
\end{equation}
\end{remark}
We would like to emphasize that the nonlinearity $A$ can go beyond Uhlenbeck category.
The nonlinearity $A$ under consideration behaves like the $p$-Laplacian operator only near the infinity with respect to the gradient variable.
In addition, it is worth pointing out that the domain under consideration is the so-called Reifenberg flat domain whose boundary is so rough that even the unit normal vector cannot be well defined.

The rest of the paper is devoted to the proof of Theorem~\ref{thm:main}.
First, in Section~\ref{Sec_Elliptic} we state and prove an elliptic version of Theorem~\ref{thm:main}.
In Section~\ref{S2} we focus on the homogeneous case, i.e., the case  when $(u_0,g_0)=(0,0)$ but $f\neq 0$.
Then in Section~\ref{S3} we apply the result for homogeneous case to the inhomogeneous one and prove the main theorem.

%

\section{Elliptic systems}\label{Sec_Elliptic}

Before proving the main theorem, we first prove elliptic estimates for the sake of better readability and to introduce the strategy of the proof.
Let us consider the following elliptic system:
\begin{align}
\label{eq:elliptic_sys}
\left\lbrace \ \begin{aligned}
\divergence A(x,\nabla u) &= \divergence f &&\textrm{ in } \Omega,\\
u&=g&&\textrm{ on } \partial \Omega,
\end{aligned} \right.
\end{align}
where $g\in W^{1,p}(\Omega)$ is a suitable extension of the given boundary values.
Here the nonlinearity $A : \mathbb{R}^n \times \mathbb{R}^{n\times N}\to \mathbb{R}^{n\times N}$ is a Carath\'{e}odory mapping, satisfying the $p$-coercivity, the $(p-1)$-growth and the monotonicity conditions as follows: there exist positive constants $c_1$ and $c_2$ such that for almost all $x \in \mathbb{R}^n$ and all $\eta_1,\eta_2 \in \mathbb{R}^{n\times N}$,
\begin{align}
\label{coercivityJM_elliptic}
c_1|\eta_1|^p-c_2 &\le A(x,\eta_1)\cdot \eta_1 &&\textrm{$p$-coercivity},\\
|A(x,\eta_1)|&\le c_2(1+|\eta_1|)^{p-1}&&\textrm{$(p-1)$-growth},\label{growthJM_elliptic}\\
0&\le (A(x,\eta_1)-A(x,\eta_2))\cdot (\eta_1-\eta_2)&&\textrm{monotonicity}.\label{monotoneJM_elliptic}
\end{align}
We also assume that $A$ is asymptotically Uhlenbeck, that is, there exists a bounded nonnegative function $\Phi : [0,\infty) \to [0,\infty)$ such that
\begin{equation}\label{AH1_elliptic}
\lim_{\rho \to \infty} \Phi(\rho) = 0
\end{equation}
and
\begin{equation}\label{AH2_elliptic}
\abs{A(x,\eta) - \abs{\eta}^{p-2}\eta} \le \Phi(\abs{\eta}) \left( 1+\abs{\eta}^{p-1} \right)
\end{equation}
for almost every $x \in \mathbb{R}^n$ and all $\eta \in \mathbb{R}^{n\times N}$.

The following theorem is the main result in this section.
\begin{theorem}
\label{thm:elliptic_main}
Let $p\in (1,\infty)$, $N\in \setN$ \pk{and $A$ satisfy \eqref{Car1}--\eqref{Car2}, \eqref{coercivityJM_elliptic}--\eqref{monotoneJM_elliptic} and \eqref{AH1_elliptic}--\eqref{AH2_elliptic}.}
Then for every $q_0 \in [1,\infty)$, there exist a small constant $\delta_0 > 0$ depending on $n$, $N$, $p$, $q_0$ and $A$ such that for any $q \in [1,q_0]$ and $\delta \in (0,\delta_0]$, we have the following properties: if $\Omega\subset \setR^n$ is a bounded $(\delta,r_0)$-Reifenberg flat domain for some $r_0>0$, i.e., it satisfies Definition~\ref{def:reif}, if $f\in L^{p'q}(\Omega; \mathbb{R}^{n\times N})$ and if $g \in W^{1,pq}(\Omega; \mathbb{R}^N)$, then the solution $u$ to \eqref{eq:elliptic_sys} belongs to $W^{1,pq}(\Omega; \mathbb{R}^N)$ and the following estimate holds
\begin{align*}
\norm{\nabla u}_{L^{pq}(\Omega; \mathbb{R}^{n\times N})} &\leq c\left(1+\norm{\nabla g}_{L^{pq}(\Omega; \mathbb{R}^{n\times N})}+\norm{f}_{L^{p'q}(\Omega; \mathbb{R}^{n\times N})}^\frac{1}{p-1}\right),
\end{align*}
with a positive constant $c$ depending only on $n$, $N$, $p$, $q_0$ and on the assumptions on $A$ and $\Omega$.
\end{theorem}

To prove the above theorem, we shall consider the system with homogeneous boundary data
\begin{align}
\label{eq:elliptic_sysB}
\left\lbrace \ \begin{aligned}
\divergence A(x,\nabla v) &= \divergence f &&\textrm{ in } \Omega,\\
v&=0&&\textrm{ on } \partial \Omega,
\end{aligned} \right.
\end{align}
and the $p$-Laplacian system with homogeneous boundary data
\begin{align}
\label{eq:elliptic_sysB2}
\left\lbrace \ \begin{aligned}
\divergence(\abs{\nabla w}^{p-2} \nabla w) &= \divergence f &&\textrm{ in } \Omega,\\
w&=0&&\textrm{ on } \partial \Omega.
\end{aligned} \right.
\end{align}
For the $p$-Laplacian system \eqref{eq:elliptic_sysB2}, we have the following Calder\'{o}n-Zygmund estimate, see \cite{ByunWangZhou07} and \cite[Remark 4.5]{ByunOk16}.

\begin{lemma}
\label{lem:elliptic_sysB2}
Let $w \in W_0^{1,p}(\Omega; \mathbb{R}^N)$ be the solution of \eqref{eq:elliptic_sysB2}.
Then for every $q_0 \in [1,\infty)$, there exist a small constant $\delta_0>0$ depending on $n$, $N$, $p$ and $q_0$ such that for any $q \in [1,q_0]$ and $\delta \in (0,\delta_0]$, we have the following properties:  if $\Omega\subset \setR^n$ is a bounded $(\delta,r_0)$-Reifenberg flat domain for some $r_0>0$ and if $f\in L^{p'q}(\Omega; \mathbb{R}^{n\times N})$, then $w \in W_0^{1,pq}(\Omega; \mathbb{R}^N)$ with the estimate
\begin{align}
\label{elliptic_estiamte}
\norm{\nabla w}_{L^{pq}(\Omega; \mathbb{R}^{n\times N})} &\leq c \norm{f}_{L^{p'q}(\Omega; \mathbb{R}^{n\times N})}^\frac{1}{p-1}
\end{align}
for some positive constant $c$ depending only on $n$, $N$, $p$, $q_0$, $r_0$ and the assumptions on $\Omega$.
\end{lemma}
%

For the asymptotically Uhlenbeck system \eqref{eq:elliptic_sysB} with a sufficiently smooth domain $\Omega$, we have the following qualitative estimate.
This result is an elliptic version of the paper \cite{Bog14}, which we repeat here to introduce our new strategy for a proof.

\begin{lemma}
\label{lem:elliptic_regularity}
\pk{Let $p\in (1,\infty)$, $N\in \setN$ and $A$ satisfy \eqref{Car1}--\eqref{Car2}, \eqref{coercivityJM_elliptic}--\eqref{monotoneJM_elliptic} and \eqref{AH1_elliptic}--\eqref{AH2_elliptic}.}
Suppose that $\Omega$ is a $C^1$-domain, and let $v$ be a solution to \eqref{eq:elliptic_sysB}.
If $f\in L^{p'q}(\Omega; \mathbb{R}^{n\times N})$, then we have $v \in W_0^{1,pq}(\Omega; \mathbb{R}^n)$.
\end{lemma}

\pk{We can generalize this Calder\'{o}n-Zygmund estimate for the weak solution to the system \eqref{eq:elliptic_sysB} also for Reifenberg flat domains as follows.}

\begin{lemma}
\label{lem:elliptic_sysB}
\pk{Let $p\in (1,\infty)$, $N\in \setN$ and $A$ satisfy \eqref{Car1}--\eqref{Car2}, \eqref{coercivityJM_elliptic}--\eqref{monotoneJM_elliptic} and \eqref{AH1_elliptic}--\eqref{AH2_elliptic}.}
For every $q_0\in [1,\infty)$ there exist a small constant $\delta_0>0$ depending on $n$, $N$, $p$, $q_0$ and $A$, such that for any $q \in [1,q_0]$ and any $\delta \in (0,\delta_0]$ we have the following properties:  if $\Omega\subset \setR^n$ is a~bounded $(\delta,r_0)$-Reifenberg flat domain for some $r_0>0$ and if $f \in L^{p'q}(\Omega; \mathbb{R}^{n\times N})$, then the solution $v$ to \eqref{eq:elliptic_sysB} belongs to $W_0^{1,pq}(\Omega; \mathbb{R}^N)$ and the following estimate holds
\begin{align}
\label{elliptic_estiamte2}
\norm{\nabla v}_{L^{pq}(\Omega; \mathbb{R}^{n\times N})} &\leq c \left( 1 + \norm{f}_{L^{p'q}(\Omega; \mathbb{R}^{n\times N})}^\frac{1}{p-1} \right)
\end{align}
with a positive constant $c$ depending only on $n$, $N$, $p$, $q_0$, $r_0$, $\Omega$ and $A$.
%
\end{lemma}

\begin{proof}
We first prove the global estimate \eqref{elliptic_estiamte2} under the a priori assumption $v \in W_0^{1,pq}(\Omega; \mathbb{R}^N)$.
We rewrite the problem \eqref{eq:elliptic_sysB} as
\begin{align*}
\left\lbrace \ \begin{aligned}
\divergence(\abs{\nabla v}^{p-2} \nabla v) &= \divergence \left( f - A(x,\nabla v) + \abs{\nabla v}^{p-2} \nabla v \right) &&\textrm{ in } \Omega,\\
v&=0&&\textrm{ on } \partial \Omega.
\end{aligned} \right.
\end{align*}
Since the nonlinearity $A$ has $(p-1)$-growth, we have
\begin{equation*}
|A(x,\nabla v)| \leq c_2 (1+|\nabla v|)^{p-1},
\end{equation*}
and hence the a priori assumption $v \in W_0^{1,pq}(\Omega; \mathbb{R}^N)$ yields that
\begin{equation*}
f - A(x,\nabla v) + \abs{\nabla v}^{p-2} \nabla v \in L^{p'q}(\Omega; \mathbb{R}^{n\times N}).
\end{equation*}
Applying Lemma \ref{lem:elliptic_sysB2}, we get
\begin{align*}
\norm{\nabla v}_{L^{pq}(\Omega; \mathbb{R}^{n\times N})} & \leq c \norm{f - A(x,\nabla v) + \abs{\nabla v}^{p-2} \nabla v}_{L^{p'q}(\Omega; \mathbb{R}^{n\times N})}^\frac{1}{p-1} \\
& \leq c_0 \left( \norm{f}_{L^{p'q}(\Omega; \mathbb{R}^{n\times N})}^\frac{1}{p-1}+\norm{A(x,\nabla v) - \abs{\nabla v}^{p-2} \nabla v}_{L^{p'q}(\Omega; \mathbb{R}^{n\times N})}^\frac{1}{p-1}\right),
\end{align*}
where $c_0$ is a positive constant depending on $n$, $N$, $p$, $q_0$, $r_0$ and the the assumptions on $\Omega$.
From \eqref{AH1_elliptic}, we see that for any $\varepsilon \in (0,1)$ there exists $M=M(\varepsilon)>1$ such that
\begin{equation}\label{AH3_elliptic}
\Phi(\rho) \leq \varepsilon, \quad \forall \rho \geq M.
\end{equation}
Then \eqref{AH2_elliptic} and \eqref{AH3_elliptic} give
\begin{align*}
|A(x,\nabla v) - \abs{\nabla v}^{p-2} \nabla v| & \leq \Phi(\abs{\nabla v}) \left( 1+\abs{\nabla v}^{p-1} \right) \\
& = \Phi(\abs{\nabla v}) \left( 1+\abs{\nabla v}^{p-1} \right) \chi_{\left\lbrace |\nabla v| < M \right\rbrace} \\
& \qquad + \Phi(\abs{\nabla v}) \left( 1+\abs{\nabla v}^{p-1} \right)  \chi_{\left\lbrace |\nabla v| \geq M \right\rbrace} \\
& \leq \norm{\Phi}_{L^{\infty}} \left( 1+M^{p-1} \right) + \varepsilon \left( 1+\abs{\nabla v}^{p-1} \right) \\
& \leq \left( 1+\norm{\Phi}_{L^{\infty}} \right) \left( 1+M^{p-1} \right) + \varepsilon \abs{\nabla v}^{p-1}.
\end{align*}
Therefore, it follows that
\begin{align*}
\norm{\nabla v}_{L^{pq}(\Omega; \mathbb{R}^{n\times N})} & \leq c_0 \norm{f}_{L^{p'q}(\Omega; \mathbb{R}^{n\times N})}^\frac{1}{p-1} + c_0 \left( 2 \left( 1+\norm{\Phi}_{L^{\infty}} \right) \left( 1+M^{p-1} \right) \right)^\frac{1}{p-1} \\
& \qquad + c_0 (2\varepsilon)^\frac{1}{p-1} \norm{|\nabla v|^{p-1}}_{L^{p'q}(\Omega; \mathbb{R}^{n\times N})}^\frac{1}{p-1}  \\
& = c_0 \norm{f}_{L^{p'q}(\Omega; \mathbb{R}^{n\times N})}^\frac{1}{p-1} + c_0 \left( 2 \left( 1+\norm{\Phi}_{L^{\infty}} \right) \left( 1+M^{p-1} \right) \right)^\frac{1}{p-1} \\
& \qquad + c_0 (2\varepsilon)^\frac{1}{p-1} \norm{\nabla v}_{L^{pq}(\Omega; \mathbb{R}^{n\times N})}.
\end{align*}
If we take $\varepsilon \in (0,1)$ so small that $c_0 (2\varepsilon)^\frac{1}{p-1} \leq \frac{1}{2}$, then we finally obtain the estimate \eqref{Th:ap2} under the a priori assumption $v \in W_0^{1,pq}(\Omega; \mathbb{R}^N)$.

Let us complete the proof of \eqref{elliptic_estiamte2} by removing the assumption $v \in W_0^{1,pq}(\Omega; \mathbb{R}^N)$.
We shall use an approximation argument as follows.
From Lemma 4.2 in \cite{ByunWang07} (see also Section 5 in \cite{ByunOkRyu13}) and a standard approximation of a Lipschitz domain by smooth domains, there exists a sequence of smooth domains $\left\lbrace \Omega_k \right\rbrace_{k=1}^{\infty}$ with the uniform $(\delta,r_0)$-Reifenberg flatness property such that $\Omega_k \subset \Omega_{k+1} \subset \Omega$ for all $k \in \mathbb{N}$ and
\begin{equation}
\label{domain_approx}
d_H (\partial \Omega_k, \partial \Omega) \longrightarrow 0 \quad \text{as} \ \, k \to \infty,
\end{equation}
where $d_H$ is the Hausdorff distance. \seb{Moreover, the approximation is achieved such that the constant in \eqref{elliptic_estiamte} is independent of $k$.}
We extend $f \in L^{p'q}(\Omega; \mathbb{R}^{n\times N})$ to zero outside $\Omega$ and consider a mollifying kernel $\varphi \in C_0^{\infty}(B_1)$ satisfying $\int_{B_1} \varphi \dx = 1$.
For $\epsilon>0$, we set
\begin{equation*}
\varphi_{\epsilon}(x) = \frac{1}{\epsilon^n} \varphi \left( \frac{x}{\epsilon} \right)
\end{equation*}
and define the mollifications
\begin{equation*}
f_{\epsilon}(x) := (f \ast \varphi_{\epsilon}) (x) \equiv \int_{\mathbb{R}^n} \varphi_{\epsilon} (y) f(x-y) \dy.
\end{equation*}
Then it is clear that $f_{\epsilon} \in C^{\infty}(\mathbb{R}^n; \mathbb{R}^{n\times N})$ and that $f_{\epsilon} \rightarrow f$ in $L^{p'q}(\Omega; \mathbb{R}^{n\times N})$ (see \cite[Appendix C]{Evans10}).
Furthermore, for each $k \in \mathbb{N}$, there exists $0 < \epsilon_k < d_H (\partial \Omega_k, \partial \Omega)$ such that
\begin{equation}
\label{mollification_norm}
\norm{f_{\epsilon_k}}_{L^{p'q}(\Omega_k; \mathbb{R}^{n\times N})} \leq \norm{f}_{L^{p'q}(\Omega; \mathbb{R}^{n\times N})}.
\end{equation}

We now consider the unique weak solution $v_k$ to the problem
\begin{align}
\label{eq:elliptic_sysB_approx}
\left\lbrace \ \begin{aligned}
\divergence A(x, \nabla v_k) &= \divergence f_{\epsilon_k} &&\textrm{ in } \Omega_k,\\
v_k&=0&&\textrm{ on } \partial \Omega_k.
\end{aligned} \right.
\end{align}
By Lemma \ref{lem:elliptic_regularity}, we know that $\nabla v_k \in L^{pq}(\Omega_k; \mathbb{R}^{n\times N})$ qualitatively.
Therefore, it follows from above and \eqref{mollification_norm} that
\begin{align*}
\norm{\nabla v_k}_{L^{pq}(\Omega_k; \mathbb{R}^{n\times N})} &\leq c \left(1+\norm{f_{\epsilon_k}}_{L^{p'q}(\Omega_k; \mathbb{R}^{n\times N})}^\frac{1}{p-1}\right) \\
& \leq c \left(1+\norm{f}_{L^{p'q}(\Omega; \mathbb{R}^{n\times N})}^\frac{1}{p-1}\right),
\end{align*}
where the constant $c$ is independent of $k$ due to the uniform Reifenberg flatness of the domain approximation.
We next let $\widetilde{v_k}$ be the zero extension of $v_k$ from $\Omega_k$ to $\Omega$, that is,
\begin{equation*}
\widetilde{v_k}(x) = \left\lbrace \begin{array}{cl}
v_k(x) & \text{if} \ x \in \Omega_k, \\
0 & \text{if} \ x \in \Omega \setminus \Omega_k.
\end{array} \right.
\end{equation*}
Then we obtain that $\widetilde{v_k} \in W_0^{1,pq}(\Omega; \mathbb{R}^N)$ with the estimate
\begin{equation}
\label{elliptic_uniform_estimate}
\norm{\nabla \widetilde{v_k}}_{L^{pq}(\Omega; \mathbb{R}^{n\times N})} = \norm{\nabla v_k}_{L^{pq}(\Omega_k; \mathbb{R}^{n\times N})} \leq c \left(1+\norm{f}_{L^{p'q}(\Omega; \mathbb{R}^{n\times N})}^\frac{1}{p-1}\right).
\end{equation}
Since $\left\lbrace \widetilde{v_k} \right\rbrace_{k=1}^{\infty}$ is uniformly bounded in $W_0^{1,pq}(\Omega; \mathbb{R}^N)$, there exist a subsequence, which we still denote by $\left\lbrace \widetilde{v_k} \right\rbrace_{k=1}^{\infty}$, and a function $\widetilde{v} \in W_0^{1,pq}(\Omega; \mathbb{R}^N)$ such that
\begin{equation}\label{elliptic_subseq_converge}
\left\lbrace \begin{array}{cl}
\nabla \widetilde{v_k} \rightharpoonup \nabla \widetilde{v} & \text{weakly in } L^{pq}(\Omega; \mathbb{R}^{n\times N}) \\
\widetilde{v_k} \rightarrow \widetilde{v} & \text{strongly in } L^{pq}(\Omega; \mathbb{R}^N)
\end{array} \right.
\end{equation}
as $k \to \infty$.
Then we conclude from \eqref{domain_approx}, \eqref{elliptic_uniform_estimate} and \eqref{elliptic_subseq_converge} that $\widetilde{v}$ is a solution to \eqref{eq:elliptic_sysB} with the estimate
\begin{equation*}
\norm{\nabla \widetilde{v}}_{L^{pq}(\Omega; \mathbb{R}^{n\times N})} \leq c \left(1+\norm{f}_{L^{p'q}(\Omega; \mathbb{R}^{n\times N})}^\frac{1}{p-1}\right).
\end{equation*}
Finally we see from the uniqueness of the weak solution to the problem \eqref{eq:elliptic_sysB} that $\widetilde{v}=v$, and the proof is complete.
\end{proof}

\begin{proof}[The proof of Theorem \ref{thm:elliptic_main}]
Assume that we fixed $q_0$, $\delta_0$ and $r_0$ according to Theorem~\ref{thm:elliptic_main}.
Next, we consider arbitrary given $q\in [1,q_0]$, $f\in L^{p'q}(\Omega; \mathbb{R}^{n\times N})$ and $g \in W^{1,pq}(\Omega; \mathbb{R}^N)$.
Let $v := u-g$.
From the trivial inequality
$$
\norm{\nabla u}_{L^{pq}(\Omega; \mathbb{R}^{n\times N})} \leq \norm{\nabla v}_{L^{pq}(\Omega; \mathbb{R}^{n\times N})} + \norm{\nabla g}_{L^{pq}(\Omega; \mathbb{R}^{n\times N})},
$$
it is enough to prove that
\begin{equation}
\label{goal_elliptic}
\norm{\nabla v}_{L^{pq}(\Omega; \mathbb{R}^{n\times N})} \leq c\left(1+\norm{\nabla g}_{L^{pq}(\Omega; \mathbb{R}^{n\times N})}+\norm{f}_{L^{p'q}(\Omega; \mathbb{R}^{n\times N})}^\frac{1}{p-1}\right)
\end{equation}
with a constant $c$ being independent of $f$ and $g$.

Since $u$ is given as $u=v+g$, we can rewrite the problem \eqref{eq:elliptic_sys} as
\begin{align}
\label{eq:elliptic_sysC}
\left\lbrace \ \begin{aligned}
\divergence A(x,\nabla v) &= \divergence \left( A(x,\nabla v) - A(x,\nabla u) + f \right) &&\textrm{ in } \Omega,\\
v&=0&&\textrm{ on } \partial \Omega,
\end{aligned} \right.
\end{align}
Since $v$ is a solution with zero boundary data, we can apply  Lemma~\ref{lem:elliptic_sysB} to obtain that for arbitrary  $\tilde{q}\in [1,q] \subset [1,q_0]$
\begin{equation}\label{basis_elliptic}
\begin{split}
\int_\Omega \abs{\nabla v}^{p\tilde{q}} \dx & \leq c \left( 1+ \int_\Omega \abs{A(x, \nabla v)-A(x, \nabla u) +f}^{p'\tilde{q}} \dx \right) \\
& \leq c \left( 1+ \int_\Omega \abs{A(x, \nabla v)-A(x, \nabla u)}^{p'\tilde{q}} \dx + \int_\Omega \abs{f}^{p'\tilde{q}} \dx \right)
\end{split}
\end{equation}
whenever the right hand side is finite.
The last term is bounded for any $\tilde{q}\le q$, so we shall justify that also \pk{the first integral on the right hand side of \eqref{basis_elliptic}} is bounded and that it can be estimated uniformly.
Let us observe the following simple algebraic inequality, which is a direct consequence of the elliptic version of \eqref{LC}
\begin{equation}\label{algebra_elliptic}
\begin{split}
\abs{A(x,\nabla v)-A(x,\nabla u)}&\leq c(\abs{\nabla u}+\abs{\nabla v})^{p-2}\abs{\nabla (u-v)}+c\\
&\le c(\abs{\nabla (u-v)}+\abs{\nabla v})^{p-2}\abs{\nabla g}+c\\
&=c(\abs{\nabla v}+\abs{\nabla g})^{p-2}\abs{\nabla g}+c
\end{split}
\end{equation}
with a constant $c$ possibly varying line to line but being independent of $u$, $v$ and $g$.

First, we start with the case $p\in (1,2]$.
It follows from \eqref{algebra_elliptic} that
\begin{equation*}
\begin{split}
\abs{A(x,\nabla v)-A(x,\nabla u)}&\leq c\abs{\nabla g}^{p-1}+c.
\end{split}
\end{equation*}
Hence, substituting this inequality into \eqref{basis_elliptic} and using also the fact that $\Omega$ is bounded, we obtain
\begin{equation*}
\int_\Omega \abs{\nabla v}^{p\tilde{q}} \dx \leq c \left( 1+ \int_\Omega \abs{\nabla g}^{p\tilde{q}} \dx + \int_\Omega \abs{f}^{p'\tilde{q}} \dx \right)
\end{equation*}
and \eqref{goal_elliptic} then directly follows.

Thus, it remains to discuss the case $p\in (2,\infty)$. Formally, i.e., in case we would know that $v\in W_0^{1,\pk{p\tilde{q}}}$,
we could use the Young inequality in \eqref{algebra_elliptic} to observe that
\begin{equation}\label{nic_elliptic}
\begin{aligned}
\abs{A(x,\nabla v)-A(x,\nabla u)}^{p'\tilde{q}}&\leq c\abs{\nabla v}^{\frac{(p-2)p \tilde{q}}{p-1}}\abs{\nabla g}^{p'\tilde{q}}+c\abs{\nabla g}^{p\tilde{q}}+c\\
&\leq \varepsilon \abs{\nabla v}^{p\tilde{q}} + c(\varepsilon)\abs{\nabla g}^{p\tilde{q}}+c.
\end{aligned}
\end{equation}
This inequality, used in \eqref{basis_elliptic} together with the fact that $p\ge 2$ (and so $p'\le p$), then leads to
\begin{equation}
\label{basisq_elliptic}
\int_\Omega \abs{\nabla v}^{p\tilde{q}} \dx \leq c\varepsilon \int_\Omega \abs{\nabla v}^{p\tilde{q}} \dx + c(\varepsilon) \left( 1+ \int_\Omega \abs{\nabla g}^{p\tilde{q}} \dx + \int_\Omega \abs{f}^{p'\tilde{q}} \dx \right).
\end{equation}
Hence choosing $\varepsilon>0$ sufficiently small, and assuming that the right hand side is finite, i.e., $v\in W_0^{1,p\tilde{q}}$, we can use the above inequality and absorb the term involving $\nabla v$ to the left hand side, to obtain
\begin{equation}
\label{basisqq_elliptic}
\int_\Omega \abs{\nabla v}^{p\tilde{q}} \dx \leq c \left( 1+ \int_\Omega \abs{\nabla g}^{p\tilde{q}} \dx + \int_\Omega \abs{f}^{p'\tilde{q}} \dx \right).
\end{equation}
However, to make this procedure rigorous, we would need to know a~priori that $v\in W_0^{1,p\tilde{q}}$,
which is not the case here. Therefore, we proceed more carefully.

First, we define
$$
q_1:=\frac{p-1}{\frac1q+p-2}
$$
and set $\tilde{q}=q_1$ in \eqref{basis_elliptic}.
Since
\[
\frac{1}{1+(p-2)q}+\frac{(p-2)}{\frac1q+(p-2)}=1,
\]
we can deduce from \eqref{algebra_elliptic} with the help of the Young inequality that (using also the fact that $q_1\le q$)
$$
\begin{aligned}
\abs{A(x,\nabla v)-A(x,\nabla u)}^{p'q_1}&\le  c\abs{\nabla v}^\frac{(p-2)p q_1}{p-1}\abs{\nabla g}^\frac{pq_1}{p-1}+c|\nabla g|^{pq_1} +c\\
 &= c\abs{\nabla v}^\frac{(p-2)p }{\frac1q+p-2}\abs{\nabla g}^\frac{pq}{1+q(p-2)}+c|\nabla g|^{pq} +c \\
& \leq c\left(1+\abs{\nabla v}^p+\abs{\nabla g}^{pq}\right).
\end{aligned}
$$
Using this inequality in \eqref{basis_elliptic} with
$\tilde{q}:=q_1$ and the fact
$p'\le p$ (since $p\ge 2$), we have
\begin{equation}
\label{basis1_elliptic}
\int_\Omega \abs{\nabla v}^{pq_1} \dx \leq c \left( 1+ \int_\Omega \abs{\nabla v}^{p} \dx + \int_\Omega \abs{\nabla g}^{pq_1} \dx + \int_\Omega \abs{f}^{p'q_1} \dx \right) < \infty.
\end{equation}
Since $|\nabla v|^p$ is integrable, we see that we improve the integrability of $\nabla v$ to $L^{p q_1}$.
Hence, we can now use \eqref{basisq_elliptic}, which has the right hand side finite and conclude the \pk{a~priori} estimate \eqref{basisqq_elliptic} rigorously.

Next, we continue iteratively. We find for
$$
q_k:=\frac{(p-1)q_{k-1}}{\frac{q_{k-1}}{q}+p-2}
$$
that $q_k<q_{k+1}<q$ and $q_k \to q$ as $k\to \infty$. Expressing now the corresponding dual exponents as
\[
\frac{1}{\big(\frac{q_{k-1}}{q}+(p-2)\big)\frac{q}{q_{k-1}}}+\frac{(p-2)}{\frac{q_{k-1}}q+(p-2)}=1,
\]
we deduce that
 \[
 \abs{\nabla v}^\frac{(p-2)p q_k}{p-1}\abs{\nabla g}^\frac{pq_k}{p-1}
 = \abs{\nabla v}^\frac{(p-2)pq_{k-1} }{\frac{q_{k-1}}{q}+p-2}\abs{\nabla g}^\frac{p q_{k-1}}{\frac{q_{k-1}}{q}+p-2}\leq \abs{\nabla v}^{pq_{k-1}}+\abs{\nabla g}^{pq}.
\]
Hence, combining this inequality with \eqref{algebra_elliptic} yields that (using $q_k\le q$ again)
\begin{equation*}
\begin{split}
\abs{A(x,\nabla v)-A(x,\nabla u)}^{p'q_k}
&\le c(\abs{\nabla v}^{(p-2)p'q_k}\abs{\nabla g}^{p'q_k} +\abs{\nabla g}^{p q_k}+1)\\
&\le c\left(1+\abs{\nabla v}^{pq_{k-1}}+\abs{\nabla g}^{pq}\right).
\end{split}
\end{equation*}
Since $\abs{\nabla v}\in L^{pq_{k-1}}$, we can use the above inequality in \eqref{basis_elliptic} to conclude improved integrability result $\abs{\nabla v} \in L^{pq_k}$.
Consequently, we can now use \eqref{basisq_elliptic} with $\tilde{q}:=q_k$, which has now the finite right hand side and to get the uniform estimate \eqref{basisqq_elliptic} with $\tilde{q}:=q_k$, i.e.,
\begin{equation}
\label{basisqqq_elliptic}
\int_\Omega \abs{\nabla v}^{pq_k} \dx \leq c \left( 1+ \int_\Omega \abs{\nabla g}^{pq_k} \dx + \int_\Omega \abs{f}^{p'q_k} \dx \right).
\end{equation}
\pk{Note that the constant $c>0$ in the previous estimate may be chosen independent of $k$}. Since $q_k \to q$ as $k\to \infty$, we can now let $k\to \infty$ in \eqref{basisqqq_elliptic}, which gives the desired conclusion \eqref{goal_elliptic}. This completes the proof.
\end{proof}

\section{Parabolic Systems: Homogeneous boundary and initial data}\label{S2}
In the following two sections we will need the following notations for parabolic intrinsic cylinders:
\[
Q^\lambda_r=Q^\lambda_r(t,x)=(t-\lambda^{2-p}r^2,t)\times B_r(x).
\]
If $p=2$ we have the standard parabolic cylinders 
\[
Q_r=Q_r(t,x)=(t-r^2,t)\times B_r(x).
\]
In this section we consider the following system:
\begin{align}
\label{eq:sysB}
\left\lbrace \ \begin{aligned}
  \partial_t v-\divergence A(t,x, \nabla v) &= -\divergence
  f &&\textrm{ in }(0,T)\times \Omega,\\
  v&=0&&\textrm{ on }(0,T)\times \partial \Omega,\\
  v(0,\cdot)&=0&&\textrm{ in } \Omega.
\end{aligned} \right.
\end{align}
To obtain a result \pk{for} the above system, we shall consider the following parabolic $p$-Laplacian system:
\begin{align}
\label{eq:sysB2}
\left\lbrace \ \begin{aligned}
  \partial_t w-\divergence(\abs{\nabla w}^{p-2} \nabla w) &= -\divergence
  f &&\textrm{ in }(0,T)\times \Omega,\\
  w&=0&&\textrm{ on }(0,T)\times \partial \Omega,\\
  w(0,\cdot)&=0&&\textrm{ in } \Omega.
\end{aligned} \right.
\end{align}
\pk{It} can be extended to the following system
\begin{align}
\label{eq:sysinfty}
\left\lbrace \ \begin{aligned}
  \partial_t w-\divergence(\abs{\nabla w}^{p-2} \nabla w) &= -\divergence
  (f\chi_{[0,T]}) &&\textrm{ in }(-\infty,\infty)\times \Omega,\\
  w&=0&&\textrm{ on }(-\infty,\infty)\times \partial \Omega,\\
  w&=0&&\textrm{ in } (-\infty,0]\times \Omega.
\end{aligned} \right.
\end{align}
Clearly, we find the following natural estimate:
\begin{align}
\sup_{-\infty < t < \infty} \int_\Omega \abs{w(t,\cdot)}^2 \dx +\int_{-\infty}^\infty\int_\Omega \abs{\nabla w}^p\dx\dt \leq c\int_0^T\int_\Omega\abs{f}^{p'}\dx\dt,
\end{align}
see for instance \cite{AceMin07, ByunOk16SIAM}.
From \cite{AceMin07,ByunOk16SIAM} one can deduce the following intrinsic local estimate:
For any $1\leq q_0<\infty$, there exists $R_0 = R_0(n,N,p,q_0) \in (0,1)$ such that for any $z_0 = (t_0,x_0) \in [0,T) \times \overline{\Omega}$ and any parabolic cylinder $Q_{2R}(z_0) \subset (-T,2T)\times \Omega$ with $0 < R \leq \frac{1}{2} R_0$ and $1 \leq q \leq q_0$,
\begin{equation}
\label{local_est}
\dint_{Q_R(z_0) \cap \Omega_T} |\nabla w|^{pq} \dz \leq c \left\lbrace \left( \dint_{Q_{2R}(z_0) \cap \Omega_T} |\nabla w|^p \dz \right)^q + \dint_{Q_{2R}(z_0) \cap \Omega_T} \left[ |f|^{p'q} + 1 \right] \dz \right\rbrace^{d},
\end{equation}
for some positive constant $c$ depending only on $n$, $N$, $p$ and
$q_0$, where $$d =
\begin{cases}
\frac{p}{2} & \text{if} \quad p\geq 2,\\
\frac{2p}{(n+2)p-2n} & \text{if} \quad \frac{2n}{n+2}<p<2,
\end{cases}
$$ is the scaling deficit constant. It implies immediately the following Lemma.

\begin{lemma}
\label{lem:scal}
Let $Q_R$ be a parabolic cylinder and $0 < R \leq \frac{1}{2} R_0$. Assume that for some $q\in [1,q_0]$,
\begin{align}
\bigg(\dint_{Q_{2R}^\lambda(z_0)} |\nabla w|^p\chi_\Omega \dz \bigg)^q
+\dint_{Q_{2R}^\lambda (z_0)}  |f|^{p'q}\chi_{Q_T}\dz \leq K\lambda^{pq}.
\end{align}
Then
\begin{align}
\dint_{\pk{Q_R^\lambda}(z_0) \cap \Omega_T} |\nabla w|^{pq}\chi_\Omega \dz\leq c\lambda^{pq},
\end{align}
where the constant $c$ just depends on $n$, $N$, $p$, $q_0$, $\Omega$ and $K$.
\end{lemma}
\begin{proof}
The proof is by the application of \eqref{local_est} to
\[
\tilde{w}(s,y)=\frac{w(\lambda^{2-p}R^2(s+t_0),R(y+x_0))}{R\lambda},
\]
which is (\seb{on its domain of definition}) a solution to
\begin{align}
\partial_t\tilde{w}-\divergence(\abs{\nabla \tilde{w}}^{p-2}\nabla \tilde{w}) = \divergence \tilde{f},
\end{align}
where
\[
\tilde{f}(s,y)=\frac{f(\lambda^{2-p}R^2(s+t_0),R(y+x_0))}{\lambda^{p-1}}.
\]
Indeed, we find that in this case (on the reflected system)
\begin{align}
\bigg(\dint_{Q_{2}(0)} |\nabla \tilde w|^p \dz \bigg)^q
+\dint_{Q_{2}(0)}  |\tilde f|^{p'q}\dz \leq K,
\end{align}
which implies an estimate that is independent of $\lambda$, and hence the result follows.
\end{proof}

\begin{remark}
\label{rem:fullsp}
The above scaling implies that if we have the following solution on \pk{the whole space}
\begin{align}
\label{eq:sys-whole-space}
\left\lbrace \ \begin{aligned}
  \partial_t w-\divergence(\abs{\nabla w}^{p-2} \nabla w) &= -\divergence
  f &&\textrm{ in }(0,\infty)\times \setR^n,\\
  w(0,\cdot)&=0&&\textrm{ in } \setR^n,
\end{aligned} \right.
\end{align}
we find the homogeneous estimate
\begin{equation}\label{Le:ap2}
\norm{w}_{X^{p,q}} \leq c\norm{f}_{L^{p'q}((0,\infty); L^{p'q}(\setR^n; \mathbb{R}^{n\times N}))}^\frac{1}{p-1}.
\end{equation}
The same estimate holds on the half space problem. Indeed for $H^+ = \mathbb{R}^n \cap \{x_n>0\}$,
\pk{the estimate holds for  solutions to}
\begin{align}
\label{eq:sys-half-space}
\left\lbrace \ \begin{aligned}
  \partial_t w-\divergence(\abs{\nabla w}^{p-2} \nabla w) &= -\divergence
  f &&\textrm{ in }(0,\infty)\times H^+,\\
    w&=0&&\textrm{ on } (0,\infty)\times \partial H^+,
  \\
  w(0,\cdot)&=0&&\textrm{ in } H^+,
\end{aligned} \right.
\end{align}
as we can reflect these solutions to the full space.
\end{remark}

We now state the main theorem in this section.

\begin{theorem}
\label{thm:plaphom}
Let $p\in (\frac{2n}{n+2},\infty)$ and $N \in \setN$. For every $q_0\in [1,\infty)$, there exist a small constant $\delta_0>0$ depending on $n$, $N$, $p$ and $q_0$ such that for any $q \in [1,q_0]$ and $\delta \in (0,\delta_0]$, we have the following properties:
if $\Omega\subset \setR^n$ is a bounded $(\delta,r_0)$-Reifenberg flat domain for some $r_0>0$, i.e., it satisfies Definition~\ref{def:reif} and if $f\in L^{p'q}(0,T; L^{p'q}(\Omega; \mathbb{R}^{n\times N}))$,
 then the solution $w$ to \eqref{eq:sysinfty} belongs to $X^{p,q}_0$ and the following estimate holds
\begin{align}\label{Th:hom}
\norm{w}_{X^{p,q}}
&\leq c\norm{f}_{L^{p'q}(0,T; L^{p'q}(\Omega; \mathbb{R}^{n\times N}))}^\frac{1}{p-1}+c\norm{f}_{L^{p'}(0,T; L^{p'}(\Omega; \mathbb{R}^{n\times N}))}^{\frac{2+p(q-1)}{2(p-1)q}}
\end{align}
with a constant $c$ depending only on $n$, $N$, $p$, $q_0$, $r_0$, \seb{the assumptions on} $\Omega$ and $T$.
In addition, we have
\begin{align}
\label{eq:hom1}
\int_0^T\int_\Omega \abs{\nabla w}^{pq}\dz\leq c\int_0^T\int_\Omega \abs{f}^{p'q}\dz+c\bigg(\int_0^T\int_\Omega \abs{\nabla w}^{p}\dz\bigg)^{1+\frac{p(q-1)}{2}}.
\end{align}
\end{theorem}

\begin{remark}
The famous scaling deficit that is usually necessary in order to get homogeneous estimates,
as is widely known in the frame work of evolutionary $p$-Laplace, pops up in the theorem above in a remarkable form.
The above theorem shifts the scaling deficit to the lower order term.
The homogeneity of the estimate with respect to the right hand side is necessary in order to include inhomogeneous boundary data.
Observe that the case $p=2$ reduces to the classic parabolic estimates for the heat equation.
Moreover, in the case $p\in (\frac{2n}{n+2},2]$, one finds that $1+\frac{p}{2}(q-1)\leq q$ and hence
\begin{align}
\label{eq:hom-ple2}
\int_0^T\int_\Omega \abs{\nabla w}^{pq} \dz\leq c\int_0^T\int_\Omega \abs{f}^{p'q} \dz+c\bigg(\int_0^T\int_\Omega \abs{\nabla w}^{p} \dz\bigg)^{q}+c,
\end{align}
which possesses no scaling deficit.
\end{remark}

\begin{proof}[Proof of Theorem~\ref{thm:plaphom}]
  First of all, we may use teh fact that we can extend $w$ as a global in time solution to $f\equiv f\chi_{Q_T}$ with related uniform estimates. 
We start by fixing $M$ and $M_1$ such that
\[
\norm{f}_{L^{p'q}([0,\infty)\times \Omega)}=M
\quad \text{ and } \quad \norm{\nabla w}_{L^{p}([0,\infty)\times \Omega)}=M_1.
\]
Let us first assume that
\[
\max\set{M_1,M}\leq 1.
\]
Now we introduce the parabolic maximal operator as
 \[
 \Mcal (g)(t,x):=\sup_{r>0}\dashint_{t-r^2}^{t+r^2}\dashint_{B_r(x)}\abs{g(s,y)}\dy\ds,
 \]
and we define the upper level sets \pk{for $m>0$ as}
\[
O^m=\Bigset{z\in \setR^{n+1}\, :\, \Mcal\Big(\abs{\nabla w}^p\chi_{\Omega}+\frac{\abs{f\chi_{Q_T}}^{p'q}}{m^{p(q-1)}}\Big)(z)>m^p}.
\]
Then the \Calderon-Zygmund covering (\cite{BreDieSch13, DieRuzWol10}) implies the following:
There are parabolic cylinders $Q_i\subset O^m$ such that $2Q_i$ have finite intersection, meaning that there is $N_0\in\mathbb N$ such that all $2Q_i$ can be placed into $N_0$ classes so that the sets $2Q_i$ in each class are disjoint.
\[
 \dashint_{2Q_i}\abs{\nabla w}^p\chi_\Omega+\frac{\abs{f\chi_{Q_T}}^{\pk{p}q}}{m^{p'(q-1)}}\dz \leq  c m^p
\]
and
\[
O^m\subset \bigcup_i Q_i.
\]
We define $r_i$ as the radius of $Q_i$.
By the weak-$L^1$ estimate, we have
\[
m^p \sum_i \abs{2Q_i}\leq m^p 2^{4n+4}\abs{O^m}\leq c.
\]
It follows that
\[
r_i^{n+2}\leq \frac{c}{m^p},
\]
which implies that $r_i\in (0,r_0)$ if $m$ is chosen large enough depending only on $n,R_0$. \seb{This is the point, where the Reifenberg flatness is used.}
Moreover, since
\[
\dashint_{2Q_i}\chi_{Q_T}\abs{f}^{p'q}\dz+ \bigg(\dashint_{2Q_i}\abs{\nabla w}^p\chi_\Omega\dz\bigg)^q\leq cm^{\pk{pq}},
\]
we find from Lemma~\ref{lem:scal} with $\lambda=1$ that
\[
\dashint_{Q_i}\abs{\nabla w}^{pq}\chi_\Omega\dz\leq c(m) {m^q}.
\]
Using the above we estimate, we get
\begin{align*}
\int_{0}^\infty\int_\Omega \abs{\nabla w}^{pq}\dz&\leq \int_{\pk{(O^m)^c\cap\left([0,\infty)\times\Omega\right)}}\abs{\nabla w}^{pq}\dz+\int_{O^m\cap \left([0,\infty)\times\Omega\right)}\abs{\nabla w}^{pq}\dz
\\
&\leq m^{p(q-1)}\int_{(O^m)^c\cap\left([0,\infty)\times\Omega\right)}\abs{\nabla w}^{p}\dz+\sum_i \int_{Q_i\cap Q_T}\abs{\nabla w}^{pq}\dz
\\
&\leq m^{p(q-1)}\int_{(O^m)^c\cap\left([0,\infty)\times\Omega\right)}\abs{\nabla w}^{p}\dz+c(m)\sum_i\abs{Q_i}m^q
\\
&\leq m^{p(q-1)}\int_{(O^m)^c\cap\left([0,\infty)\times\Omega\right)}\abs{\nabla w}^{p}\chi_\Omega\dz+c(m)\sum_i\abs{Q_i}\pk{erase} m^{p}
\\
&\leq c(R_0).
\end{align*}
This implies the following:
\begin{align}
\label{eq:M=1}
\text{If }\max\set{M_1,M}\leq 1, \ \text{then }\int_0^\infty \int_\Omega \abs{\nabla w}^{pq} \dz\leq c(R_0).
\end{align}

For general $M, M_1\neq 0$, we use the fact that we may rescale as follows. We know that
  \[
\tilde{w}(s,y)=\frac{w(\lambda^{2-p}s,y)}{\lambda}
\]
 is a solution to
\begin{align}
\partial_t\tilde{w}-\divergence(\abs{\nabla \tilde{w}}^{p-2}\nabla \tilde{w}) = \divergence \tilde{f},
\end{align}
where
\[
\tilde{f}(s,y)=\frac{f(\lambda^{2-p}s,y)}{\lambda^{p-1}}.
\]
Fixing $\lambda>0$ so that
\[
\max\set{M^{p'q}\lambda^{p-2-pq},M_1^p\lambda^{-2}}=1,
\]
we obtain
\[
\int_0^\infty\int_\Omega\abs{\tilde{f}}^{p'q} \dy \ds=\int_0^\infty\int_\Omega\abs{f}^{p'q}\lambda^{-pq} \dy \, \lambda^{p-2}\, \dt\leq M^{p'q}\lambda^{p-2-pq} \leq 1,
\]
\[
\int_0^\infty\int_\Omega \abs{{\nabla\tilde{w} }}^{p}\dy\ds=\int_0^\infty\int_\Omega\abs{\nabla w}^{p}\lambda^{-p} \dy \, \lambda^{p-2}\, \dt\leq M_1^{p}\lambda^{-2} \leq 1.
\]
Then we discover from \eqref{eq:M=1} that
\begin{equation}
\label{eq:scaled_estimate}
c(R_0)\geq \int_0^\infty\int_\Omega\abs{\nabla \tilde{w}}^{pq}\, \dy\ds=\int_0^\infty\int_\Omega\abs{\nabla w}^{pq}\lambda^{p-2-pq}\, \dy \dt.
\end{equation}
If $\max\set{M^{p'q}\lambda^{p-2-pq},M_1^p\lambda^{-2}} = M^{p'q}\lambda^{p-2-pq} = 1$, then \eqref{eq:scaled_estimate} gives
\begin{equation}
\label{eq:hom1_case1}
\int_0^\infty\int_\Omega\abs{\nabla w}^{pq} \, \dy \dt \leq c M^{p'q} = c\int_0^\infty\int_\Omega \abs{f}^{p'q} \dz.
\end{equation}
If $\max\set{M^{p'q}\lambda^{p-2-pq},M_1^p\lambda^{-2}} = M_1^p\lambda^{-2} = 1$, then \eqref{eq:scaled_estimate} yields
\begin{equation}
\label{eq:hom1_case2}
\int_0^\infty\int_\Omega\abs{\nabla w}^{pq} \, \dy \dt \leq c M_1^{\frac{p(2+pq-p)}{2}} = c\bigg(\int_0^\infty\int_\Omega \abs{\nabla w}^{p} \dz\bigg)^{\frac{2+pq-p}{2}}.
\end{equation}
Combining \eqref{eq:hom1_case1} and \eqref{eq:hom1_case2}, we obtain the desired estimate \eqref{eq:hom1}.
Finally, \eqref{Th:hom} follows by the energy inequality and the properties of the weak time derivative.
\end{proof}

For the system \eqref{eq:sysB} with a sufficiently smooth domain $\Omega$, we have the following qualitative estimate.

\begin{lemma}
\label{lem2:main2}
Let $p\in (\frac{2n}{n+2},\infty)$, $N\in \setN$ and $A$ satisfy \eqref{Car1}--\eqref{monotoneJM} and \eqref{AH1}--\eqref{AH2}.
Suppose that $\Omega$ is a $C^1$-domain, and let $v$ be a solution to \eqref{eq:sysB}.
If $f\in L^{p'q}((0,T; L^{p'q}(\Omega; \mathbb{R}^{n\times N}))$, then we have $v \in X^{p,q}_0$.
\end{lemma}

\begin{proof}
Let us consider the extended system
\begin{align}
\label{eq:sysB_extended}
\left\lbrace \ \begin{aligned}
  \partial_t \widetilde{v}-\divergence A(t,x, \nabla \widetilde{v}) &= -\divergence \widetilde{f} &&\textrm{ in }(-T,T)\times \Omega,\\
  \widetilde{v}&=0&&\textrm{ on }(-T,T)\times \partial \Omega,\\
  \widetilde{v}(-T,\cdot)&=0&&\textrm{ in } \Omega.
\end{aligned} \right.
\end{align}
where $\widetilde{v}$, $\widetilde{f}$ are defined by
\begin{equation*}
\widetilde{v}(t,x) = \left\lbrace \begin{array}{cl}
v(t,x) & \text{if} \ 0 < t \leq T, \\
0 & \text{if} \ -T \leq t \leq 0,
\end{array} \right.
\quad \text{and} \quad
\widetilde{f}(t,x) = \left\lbrace \begin{array}{cl}
f(t,x) & \text{if} \ 0 < t \leq T, \\
0 & \text{if} \ -T \leq t \leq 0.
\end{array} \right.
\end{equation*}
Then it follows from \cite[Theorem 2.3]{Bog14} that $\widetilde{v} \in X^{p,q}_0((-T+\varepsilon,T) \times \Omega)$ for any $\varepsilon \in (0,2T)$.
Therefore, we conclude from the definition of $\widetilde{v}$ that $v \in X^{p,q}_0((0,T) \times \Omega)$.
\end{proof}

For the system \eqref{eq:sysB} with a nonsmooth domain $\Omega$, we have the following particular result.
\begin{theorem}
\label{thm:main2}
Let $p\in (\frac{2n}{n+2},\infty)$, $N\in \setN$ and $A$ satisfy \eqref{Car1}--\eqref{monotoneJM} and \eqref{AH1}--\eqref{AH2}. Then for every $q_0 \in [1,\infty)$, there exists a small constant $\delta_0>0$ depending on $n$, $N$, $p$, $q_0$ and $A$ such that for any $q \in [1,q_0]$ and any $\delta \in (0,\delta_0]$, we have the following properties: If $\Omega\subset \setR^n$ is a bounded $(\delta,r_0)$-Reifenberg flat domain for some $r_0 \in (0,1)$ and
if $f\in L^{p'q}((0,T; L^{p'q}(\Omega; \mathbb{R}^{n\times N}))$, then the unique weak solution $v$ to \eqref{eq:sysB} belongs to $X^{p,q}_0$ and the estimate
\begin{equation}\label{Th:ap2}
\norm{v}_{X^{p,q}} \leq c\left(1+\norm{f}_{L^{p'q}(0,T; L^{p'q}(\Omega; \mathbb{R}^{n\times N}))}^\frac{1}{p-1}+\norm{f}_{L^{p'}(0,T; L^{p'}(\Omega; \mathbb{R}^{n\times N}))}^{\frac{2+p(q-1)}{2(p-1)q}} \right)
\end{equation}
holds with a constant $c$ depending only on $n$, $N$, $p$, $q_0$, $r_0$, $T$ and the assumed properties of $A$ and $\Omega$.
\end{theorem}

\begin{proof}
The existence of a weak solution to \eqref{eq:sysB} follows from the theory of monotone operators or via Galerkin approximation; we refer the reader to \cite{DiB93, Lions69, Show97} and the references therein.

To prove the uniqueness, we assume that there exist two solutions $v_1$ and $v_2$ to the problem \eqref{eq:sysB}.
Notice that $v_1(t,\cdot)-v_2(t,\cdot) \in W^{1,p}_0(\Omega;\mathbb{R}^N)$ for almost all $t \in (0,T)$.
Taking this function for a test function in the weak formulation \eqref{D3}, we obtain that for almost all $t \in (0,T)$,
\begin{multline*}
\langle v_i(t), (v_1-v_2)(t) \rangle + \int_{\Omega} A(t,x,\nabla v_i(t,x)) \cdot \nabla (v_1-v_2)(t,x) \dx \\
= \int_{\Omega}f(t,x)\cdot \nabla (v_1-v_2)(t,x) \dx, \quad \forall i=1,2.
\end{multline*}
Subtracting each other and integrating over $(0,T)$, we have
\begin{multline*}
\int_0^T \langle (v_1-v_2)(t), (v_1-v_2)(t) \rangle \dt \\
+ \int_0^T \int_\Omega \left( A(t,x,\nabla v_1(t,x)) - A(t,x,\nabla v_2(t,x)) \right) \cdot \nabla \left( v_1(t,x)-v_2(t,x) \right) \dx \dt = 0.
\end{multline*}
Since $\langle (v_1-v_2)(t), (v_1-v_2)(t) \rangle \geq 0$, it follows that
\begin{equation*}
 \int_0^T \int_\Omega \left( A(t,x,\nabla v_1(t,x)) - A(t,x,\nabla v_2(t,x)) \right) \cdot \nabla \left( v_1(t,x)-v_2(t,x) \right) \dx \dt \leq 0.
\end{equation*}
The strict monotonicity of $A$ leads to $\nabla v_1 = \nabla v_2$ almost everywhere in $(0,T) \times \Omega$.
Thus $v_1 \equiv v_2$ in $(0,T) \times \Omega$, which is due to the zero trace.

We now prove the global estimate \eqref{Th:ap2} under the a priori assumption $v \in X^{p,q}_0$.
We rewrite the problem \eqref{eq:sysB} as
\begin{align*}
\left\lbrace \ \begin{aligned}
  \partial_t v-\divergence(\abs{\nabla v}^{p-2} \nabla v) &= -\divergence \left( f - A(t,x,\nabla v) + \abs{\nabla v}^{p-2} \nabla v \right) &&\textrm{ in }(0,T)\times \Omega,\\
  v&=0&&\textrm{ on }(0,T)\times \partial \Omega,
  \\
  v(0,\cdot)&=0&&\textrm{ in } \Omega.
\end{aligned} \right.
\end{align*}
Notice that \eqref{growthJM} implies
\begin{equation*}
|A(t,x,\nabla v)| \leq c_2 (1+|\nabla v|)^{p-1},
\end{equation*}
and hence the a priori assumption $v \in X^{p,q}_0$ yields that
\begin{equation*}
f - A(t,x,\nabla v) + \abs{\nabla v}^{p-2} \nabla v \in L^{p'q}((0,T; L^{p'q}(\Omega; \mathbb{R}^{n\times N})).
\end{equation*}
Applying Theorem \ref{thm:plaphom}, we get
\begin{equation}
\begin{split}
\int_0^T\int_\Omega \abs{\nabla v}^{pq}\dx \dt & \leq c\int_0^T\int_\Omega \abs{f - A(t,x,\nabla v) + \abs{\nabla v}^{p-2} \nabla v}^{p'q}\dx \dt \\
& \quad +c\bigg(\int_0^T\int_\Omega \abs{\nabla v}^{p}\dx \dt\bigg)^{1+\frac{p(q-1)}{2}} \\
& \leq c\int_0^T\int_\Omega \abs{A(t,x,\nabla v) - \abs{\nabla v}^{p-2} \nabla v}^{p'q}\dx \dt \\
& \quad +c\int_0^T\int_\Omega \abs{f}^{p'q}\dx \dt +c\bigg(\int_0^T\int_\Omega \abs{\nabla v}^{p}\dx \dt\bigg)^{1+\frac{p(q-1)}{2}}
\end{split}\label{para_est_0}
\end{equation}
From the structural conditions \eqref{coercivityJM}-\eqref{monotoneJM}, we obtain the following energy estimate of the problem \eqref{eq:sysB}:
\begin{equation}
\label{para_est_1}
\sup_{0<t<T} \int_\Omega \abs{v(t,\cdot)}^2 \dx + \int_0^T \int_\Omega \abs{\nabla v}^p\dx\dt \leq c\int_0^T\int_\Omega\abs{f}^{p'}\dx\dt.
\end{equation}
On the other hand, we see from \eqref{AH1} that for any $\varepsilon \in (0,1)$ there exists $M=M(\varepsilon)>1$ such that
\begin{equation}\label{AH3}
\Phi(\rho) \leq \varepsilon, \quad \forall \rho \geq M.
\end{equation}
Then \eqref{AH2} and \eqref{AH3} give
\begin{equation}
\begin{split}
|A(t,x,\nabla v) - \abs{\nabla v}^{p-2} \nabla v| & \leq \Phi(\abs{\nabla v}) \left( 1+\abs{\nabla v}^{p-1} \right) \\
& = \Phi(\abs{\nabla v}) \left( 1+\abs{\nabla v}^{p-1} \right) \chi_{\left\lbrace |\nabla v| < M \right\rbrace} \\
& \qquad + \Phi(\abs{\nabla v}) \left( 1+\abs{\nabla v}^{p-1} \right)  \chi_{\left\lbrace |\nabla v| \geq M \right\rbrace} \\
& \leq \norm{\Phi}_{L^{\infty}} \left( 1+M^{p-1} \right) + \varepsilon \left( 1+\abs{\nabla v}^{p-1} \right) \\
& \leq \left( 1+\norm{\Phi}_{L^{\infty}} \right) \left( 1+M^{p-1} \right) + \varepsilon \abs{\nabla v}^{p-1}.
\end{split}\label{para_est_2}
\end{equation}
Therefore, it follows from \eqref{para_est_0}, \eqref{para_est_1} and \eqref{para_est_2} that
\begin{align*}
\int_0^T\int_\Omega \abs{\nabla v}^{pq}\dx \dt & \leq c\int_0^T\int_\Omega \abs{f}^{p'q}\dx \dt +c\bigg(\int_0^T\int_\Omega \abs{f}^{p'}\dx \dt\bigg)^{1+\frac{p(q-1)}{2}} \\
& \qquad +c |\Omega| T \left( 1+\norm{\Phi}_{L^{\infty}} \right) \left( 1+M^{p-1} \right)^{p'q} + c\varepsilon^{p'q} \int_0^T\int_\Omega \abs{\nabla v}^{pq}\dx \dt.
\end{align*}
If we take $\varepsilon \in (0,1)$ so small that $c\varepsilon^{p'q} \leq \frac{1}{2}$, then we finally obtain the estimate \eqref{Th:ap2} under the a priori assumption $v \in X^{p,q}_0$.

Let us complete the proof of \eqref{Th:ap2} by removing the assumption $v \in X^{p,q}_0$.
We shall use an approximation argument as follows.
From \cite[Lemma 4.2]{ByunWang07} with \cite[Section 5]{ByunOkRyu13} and a standard approximation of a Lipschitz domain by smooth domains, there exists a sequence of smooth domains $\left\lbrace \Omega_k \right\rbrace_{k=1}^{\infty}$ with the uniform $(\delta,r_0)$-Reifenberg flatness property such that $\Omega_k \subset \Omega_{k+1} \subset \Omega$ for all $k \in \mathbb{N}$ and
\begin{equation}
d_H (\partial \Omega_k, \partial \Omega) \longrightarrow 0 \quad \text{as} \ \, k \to \infty,
\end{equation}
where $d_H$ is the Hausdorff distance.
Let $v_k$ be the unique weak solution to the problem
\begin{align}
\label{eq:sysB_approx}
\left\lbrace \ \begin{aligned}
  \partial_t v_k-\divergence A(t,x, \nabla v_k) &= -\divergence
  f &&\textrm{ in }(0,T)\times \Omega_k,\\
  v_k&=0&&\textrm{ on }(0,T)\times \partial \Omega_k,\\
  v_k(0,\cdot)&=0&&\textrm{ in } \Omega_k.
\end{aligned} \right.
\end{align}
By Lemma \ref{lem2:main2}, we know that $v_k \in X^{p,q}_0((0,T) \times \Omega_k)$.
Therefore, it follows from the a priori global estimate that
\begin{align*}
\norm{v_k}_{X^{p,q}((0,T) \times \Omega_k)} &\leq c \left(1+\norm{f}_{L^{p'q}(0,T; L^{p'q}(\Omega_k; \mathbb{R}^{n\times N}))}^\frac{1}{p-1}+\norm{f}_{L^{p'}(0,T; L^{p'}(\Omega_k; \mathbb{R}^{n\times N}))}^{\frac{2+p(q-1)}{2(p-1)q}}\right) \\
& \leq c \left(1+\norm{f}_{L^{p'q}(0,T; L^{p'q}(\Omega; \mathbb{R}^{n\times N}))}^\frac{1}{p-1}+\norm{f}_{L^{p'}(0,T; L^{p'}(\Omega; \mathbb{R}^{n\times N}))}^{\frac{2+p(q-1)}{2(p-1)q}}\right),
\end{align*}
where the constant $c$ is independent of $k$ \seb{by the uniformity of the estimates and the construction of the approximation of the domain}.
\pk{Next,} let $\widetilde{v_k}$ be the zero extension of $v_k$ from $(0,T) \times \Omega_k$ to $(0,T) \times \Omega$, that is,
\begin{equation*}
\widetilde{v_k}(t,x) = \left\lbrace \begin{array}{cl}
v_k(t,x) & \text{if} \ (t,x) \in (0,T) \times \Omega_k, \\
0 & \text{if} \ (t,x) \in (0,T) \times ( \Omega \setminus \Omega_k).
\end{array} \right.
\end{equation*}
Then we obtain that $\widetilde{v_k} \in X^{p,q}_0((0,T) \times \Omega)$ with the estimate
\begin{equation}
\begin{split}
\norm{\widetilde{v_k}}_{X^{p,q}((0,T) \times \Omega)} & = \norm{v_k}_{X^{p,q}((0,T) \times \Omega_k)} \\
& \leq c \left(1+\norm{f}_{L^{p'q}(0,T; L^{p'q}(\Omega; \mathbb{R}^{n\times N}))}^\frac{1}{p-1}+\norm{f}_{L^{p'}(0,T; L^{p'}(\Omega; \mathbb{R}^{n\times N}))}^{\frac{2+p(q-1)}{2(p-1)q}}\right).
\end{split}\label{uniform_estimate}
\end{equation}
Hence $\left\lbrace \widetilde{v_k} \right\rbrace_{k=1}^{\infty}$ is uniformly bounded in $X^{p,q}_0((0,T) \times \Omega)$.
Also it follows from \eqref{eq:sysB_approx}, \eqref{growthJM} and \eqref{uniform_estimate} that $\left\lbrace (\widetilde{v_k})_t \right\rbrace_{k=1}^{\infty}$ is uniformly bounded in $L^{p'q}((0,T);W^{-1,p'q}(\Omega;\setR^N))$.
In view of Aubin-Lions Lemma, there exist a subsequence, which we still denote by $\left\lbrace \widetilde{v_k} \right\rbrace_{k=1}^{\infty}$, and a function $\widetilde{v} \in X^{p,q}_0((0,T) \times \Omega)$ such that
\begin{equation}\label{subseq_converge}
\left\lbrace \begin{array}{cl}
\widetilde{v_k} \rightharpoonup \widetilde{v} & \text{weakly in } L^{pq}((0,T); W_0^{1,pq}(\Omega;\setR^N)) \\
\widetilde{v_k} \rightarrow \widetilde{v} & \text{strongly in } L^{pq}((0,T); L^{pq}(\Omega;\setR^N)) \\
(\widetilde{v_k})_t \rightharpoonup (\widetilde{v})_t & \text{weakly in } L^{p'q}((0,T);W^{-1,p'q}(\Omega;\setR^N))
\end{array} \right.
\end{equation}
as $k \to \infty$.
Then we conclude from \eqref{uniform_estimate} and \eqref{subseq_converge} that $\widetilde{v}$ is a solution to \eqref{eq:sysB} with the estimate
\begin{equation*}
\norm{\widetilde{v}}_{X^{p,q}((0,T) \times \Omega)} \leq c \left(1+\norm{f}_{L^{p'q}(0,T; L^{p'q}(\Omega; \mathbb{R}^{n\times N}))}^\frac{1}{p-1}+\norm{f}_{L^{p'}(0,T; L^{p'}(\Omega; \mathbb{R}^{n\times N}))}^{\frac{2+p(q-1)}{2(p-1)q}}\right).
\end{equation*}
Finally we see from the uniqueness of the weak solution to the problem \eqref{eq:sysB} that $\widetilde{v}=v$, and the proof is complete.
\end{proof}

\begin{remark}
From Theorem \ref{thm:main2}, we see that the unique weak solution $v$ to \eqref{eq:sysB} belongs to $X^{p,q}_0$.
Hence, it follows from the proof of Theorem \ref{thm:main2} that the estimate
\begin{equation}\label{Th:ap2_2}
\int_0^T\int_\Omega \abs{\nabla v}^{pq}\dz \leq c \left\{ 1 + \int_0^T\int_\Omega \abs{f}^{p'q}\dz + \left(\int_0^T\int_\Omega \abs{\nabla v}^{p}\dz\right)^{1+\frac{p(q-1)}{2}} \right\}
\end{equation}
also holds with a constant $c$ depending only on $n$, $N$, $p$, $q_0$, $r_0$, $\Omega$, $T$ and $A$.
\end{remark}

\section{The proof of Theorem~\ref{thm:main} for general data} \label{S3}
Assume that we fixed $q_0$, $\delta_0$ and $r_0$ according to
Theorem~\ref{thm:main}. Next, we consider any $q\in [1,q_0]$, $f\in
L^{p'q}(0,T; L^{p'q}(\Omega; \mathbb{R}^{n\times N}))$ and
a suitable $L^{pq}$-couple
$(u_0,g_0)$. For this couple, we then can find an extension $g\in
X^{p,q}$ and\pk{, } since $q\ge 1$, we also know that there exists a unique
(for a given extension) weak solution $u\in X^{p,1}$ to
\eqref{eq:sysA}, where $u=v+g$ with $v\in X^{p,1}_0$ fulfilling
$v(0, \cdot)=0$. Using the obvious
inequality
$$
\norm{u}_{X^{p,q}} \le \norm{v}_{X^{p,q}} + \norm{g}_{X^{p,q}},
$$
we see that to prove Theorem~\ref{thm:main}, it is enough to show that
\begin{equation}
\begin{split}
\norm{v}_{X^{p,q}} & \le c \Bigg(1+ \norm{g}_{X^{p,q}} + \norm{f}_{L^{p'q}(0,T; L^{p'q}(\Omega;\mathbb{R}^{n\times N}))}^{\frac{1}{p-1}} \\
& \qquad \qquad \quad +\norm{g}_{X^{p,q}}^{\frac{2+p(q-1)}{2q}} +\norm{f}_{L^{p'}(0,T; L^{p'}(\Omega; \mathbb{R}^{n\times N}))}^{\frac{2+p(q-1)}{2(p-1)q}} \Bigg),
\end{split}\label{goal}
\end{equation}
with a constant $c$ being independent of $f$ and $g$. Thus, in what follows, we focus only on the proof of \eqref{goal}.

First of all, we represent $\partial_t g$ as a divergence of a vector-valued function whose norm is controlled in an appropriate space. To do so, we look for a unique solution to the following problem: find $w:\Omega \to \mathbb{R}^N$ solving
$$
\left\lbrace \ \begin{aligned}
-\divergence \left( |\nabla w(t)|^{(p'q)'-2}\nabla w \right) &= \partial_t g(t) &&\textrm{ in }\Omega,\\
w&=0 &&\textrm{ on } \Omega.
\end{aligned} \right.
$$
Here, we can use the standard monotone operator theory to obtain the existence of a unique $w\in W^{1,(p'q)'}_0(\Omega;\mathbb{R}^N)$ solving the above problem. Moreover, we know that it continuously depends on $\partial_t g$, then it is also Bochner measurable and we have an estimate
\begin{equation*}
\norm{\nabla w(t)}_{L^{(p'q)'}(\Omega;\mathbb{R}^{n \times N})}^{(p'q)'}\le c\norm{\partial_t g(t)}^{p'q}_{W^{-1,p'q}(\Omega;\mathbb{R}^N)}.
\end{equation*}
Thus, if we set $G:=|\nabla w|^{(p'q)'-2} \nabla w$, then we have that $\skp{\partial_t g(t)}{\psi}=\skp{G(t)}{\nabla \psi}$ for almost every $t\in [0,T]$ and all $\psi \in W^{1,(p'q)'}_0(\Omega; \mathbb{R}^N)$. In addition, we have an estimate
\begin{equation}\label{G:est}
\int_0^T \norm{G}_{L^{p'q}(\Omega;\mathbb{R}^{n \times N})}^{p'q}\dt \le c\int_0^T \norm{\partial_t g}^{p'q}_{W^{-1,p'q}(\Omega;\mathbb{R}^N)}\dt.
\end{equation}
Next, since $u$ is given as $u=v+g$, we can use the definition of $G$ and rewrite the problem \eqref{eq:sysA} as
\begin{align}
\label{eq:sysC}
\left\lbrace \ \begin{aligned}
  \partial_tv-\divergence A(t,x, \nabla v) &= -\partial_t g-\divergence \left(A(t,x, \nabla v)-A(t,x, \nabla u) + f \right)
  \\
  &= -\divergence \left(A(t,x, \nabla v)-A(t,x, \nabla u) +f-G \right)
  &&\textrm{ in }(0,T)\times \Omega,\\
  v&=0&&\textrm{ on } (0,T)\times \partial \Omega,\\
  v(0,\cdot)&=0&&\textrm{ in } \Omega.
\end{aligned} \right.
\end{align}
Since $v$ is a solution with zero initial and boundary data, we can
apply  Theorem~\ref{thm:main2} to
conclude that for any
$\tilde{q}\in [1,q]$ (which is surely less or equal to $q_0$),
\begin{equation}
\begin{split}
& \int_0^T\int_\Omega \abs{\nabla v}^{p\tilde{q}} \dx \dt \\
& \qquad \leq c\Bigg\{ 1+ \int_0^T\int_\Omega \abs{A(t,x, \nabla v)-A(t,x, \nabla u) +f-G}^{p'\tilde{q}} \dx \dt \\
& \qquad \qquad \qquad \qquad \qquad \qquad \qquad \qquad \qquad + \left(\int_0^T\int_\Omega \abs{\nabla v}^{p} \dx \dt \right)^{1+\frac{p(\tilde{q}-1)}{2}} \Bigg\}
\\
& \qquad \leq c\Bigg\{ 1+\int_0^T\int_\Omega \abs{A(t,x,\nabla v)-A(t,x,\nabla u)}^{p'\tilde{q}} \dx \dt \\
& \qquad \qquad \quad + \norm{g}_{X^{p,q}}^{p\tilde{q}} + \norm{f}^{p'\tilde{q}}_{L^{p'q}(0,T;L^{p'q}(\Omega; \mathbb{R}^{n\times N}))}+ \left(\int_0^T\int_\Omega \abs{\nabla v}^{p} \dx \dt \right)^{1+\frac{p(\tilde{q}-1)}{2}} \Bigg\},
\end{split}\label{basis_0}
\end{equation}
whenever the right hand side is finite. Notice here that for the second inequality, we used~\eqref{G:est}, the H\"{o}lder inequality and the fact that $\Omega$ is bounded.
Since $v=u-g$, we get
\begin{equation}
\begin{split}
\int_0^T\int_\Omega \abs{\nabla v}^{p}\dx \dt & \leq \int_0^T\int_\Omega \abs{\nabla u}^{p}\dx \dt + \int_0^T\int_\Omega \abs{\nabla g}^{p}\dx \dt \\
& \leq c \left( \int_0^T\int_\Omega \abs{f}^{p'}\dx \dt + \int_0^T\int_\Omega \abs{\nabla g}^{p}\dx \dt \right),
\end{split}\label{basis_energy}
\end{equation}
where we have used an energy estimate of the problem \eqref{eq:sysA}.
Combining \eqref{basis_0} and \eqref{basis_energy} gives
\begin{equation}
\begin{split}
& \int_0^T\int_\Omega \abs{\nabla v}^{p\tilde{q}} \dx \dt \\
& \quad \leq c\Bigg\{ 1+\int_0^T\int_\Omega \abs{A(t,x,\nabla v)-A(t,x,\nabla u)}^{p'\tilde{q}} \dx \dt \\
& \qquad \qquad + \norm{g}_{X^{p,q}}^{p\tilde{q}} + \norm{f}^{p'\tilde{q}}_{L^{p'q}(0,T;L^{p'q}(\Omega; \mathbb{R}^{n\times N}))} + \norm{g}_{X^{p,1}}^{p+\frac{p^2(\tilde{q}-1)}{2}} + \norm{f}_{L^{p'}(0,T;L^{p'}(\Omega; \mathbb{R}^{n\times N}))}^{p'+\frac{p'p(\tilde{q}-1)}{2}} \Bigg\}.
\end{split}\label{basis}
\end{equation}
Since the last four terms are
bounded for any $\tilde{q}\le q$, it suffices to show that the first integral term in \eqref{basis} is bounded in a universal way. We start with a simple algebraic
inequality, which is a direct consequence of \eqref{LC}
\begin{equation}\label{algebra}
\begin{split}
\abs{A(t,x,\nabla v)-A(t,x,\nabla u)}&\leq c(\abs{\nabla u}+\abs{\nabla v})^{p-2}\abs{\nabla (u-v)}+c\\
&\le c(\abs{\nabla (u-v)}+\abs{\nabla v})^{p-2}\abs{\nabla g}+c\\
&=c(\abs{\nabla v}+\abs{\nabla g})^{p-2}\abs{\nabla g}+c,
\end{split}
\end{equation}
with a constant $c$ possibly varying line to line but being independent of $u$, $v$ and $g$.

First, we consider the case that
$p\in (\frac{2n}{n+2},2]$. Then it follows from \eqref{algebra} that
\begin{equation*}
\begin{split}
\abs{A(t,x,\nabla v)-A(t,x,\nabla u)}&\leq c\abs{\nabla g}^{p-1}+c.
\end{split}
\end{equation*}
Hence, substituting this inequality into \eqref{basis} and using also the fact that $\Omega$ is bounded, we obtain
\begin{equation*}
\begin{split}
\int_0^T\int_\Omega \abs{\nabla v}^{p\tilde{q}} \dx \dt & \leq c\Bigg(1 +\norm{g}_{X^{p,q}}^{p\tilde{q}} + \norm{f}^{p'\tilde{q}}_{L^{p'q}(0,T;L^{p'q}(\Omega; \mathbb{R}^{n\times N}))} \\
& \qquad \qquad  + \norm{g}_{X^{p,1}}^{p+\frac{p^2(\tilde{q}-1)}{2}}
+ \norm{f}_{L^{p'}(0,T;L^{p'}(\Omega; \mathbb{R}^{n\times
N}))}^{p'+\frac{p'p(\tilde{q}-1)}{2}} \Bigg),
\end{split}
\end{equation*}
and \eqref{goal} then directly follows if for the estimate of
$\partial_t v$, one uses \eqref{eq:sysC} together with the growth
assumption on $A$, see \eqref{growthJM}.

We next discuss the case that $p\in (2,\infty)$. Formally, i.e., in case we would know that $v\in
X^{p,\tilde{q}}$, we could use the Young inequality in
\eqref{algebra} to observe that
\begin{equation}\label{nic}
\begin{aligned}
\abs{A(t,x,\nabla v)-A(t,x,\nabla u)}^{p'\tilde{q}}&\leq c\abs{\nabla v}^{\frac{(p-2)p \tilde{q}}{p-1}}\abs{\nabla g}^{p'\tilde{q}}+c\abs{\nabla g}^{p\tilde{q}}+c\\
&\leq \varepsilon \abs{\nabla v}^{p\tilde{q}} + c(\varepsilon)\abs{\nabla g}^{p\tilde{q}}+c.
\end{aligned}
\end{equation}
This inequality, used in \eqref{basis} together with the fact that $p\ge 2$ (and so $p'\le p$), then leads to
\begin{equation}
\begin{split}
\int_0^T\int_\Omega \abs{\nabla v}^{p\tilde{q}} \dx \dt &\leq c\varepsilon \int_0^T\int_\Omega \abs{\nabla v}^{p\tilde{q}} \dx \dt\\
&\quad + c(\varepsilon)\Bigg(1 +\norm{g}_{X^{p,q}}^{p\tilde{q}} + \norm{f}^{p'\tilde{q}}_{L^{p'q}(0,T;L^{p'q}(\Omega; \mathbb{R}^{n\times N}))} \\
& \qquad \qquad \qquad + \norm{g}_{X^{p,1}}^{p+\frac{p^2(\tilde{q}-1)}{2}} + \norm{f}_{L^{p'}(0,T;L^{p'}(\Omega; \mathbb{R}^{n\times N}))}^{p'+\frac{p'p(\tilde{q}-1)}{2}} \Bigg).
\end{split}\label{basisq}
\end{equation}
Hence choosing $\varepsilon>0$ sufficiently small, and assuming that
the right hand side is finite, i.e., $v\in X^{p,\tilde{q}}$, we can
use the above inequality and absorb the term involving $\nabla v$ to
the left hand side to discover that
\begin{equation}
\begin{split}
\int_0^T\int_\Omega \abs{\nabla v}^{p\tilde{q}} \dx \dt &\leq c\Bigg(1 +\norm{g}_{X^{p,q}}^{p\tilde{q}} + \norm{f}^{p'\tilde{q}}_{L^{p'q}(0,T;L^{p'q}(\Omega; \mathbb{R}^{n\times N}))} \\
& \qquad \qquad + \norm{g}_{X^{p,1}}^{p+\frac{p^2(\tilde{q}-1)}{2}} + \norm{f}_{L^{p'}(0,T;L^{p'}(\Omega; \mathbb{R}^{n\times N}))}^{p'+\frac{p'p(\tilde{q}-1)}{2}} \Bigg).
\end{split}\label{basisqq}
\end{equation}
We now remove the a~priori assumption that $v\in X^{p,q}$.
Write
$$
q_1:=\frac{p-1}{\frac1q+p-2}
$$
and we want to set $\tilde{q}=q_1$ in \eqref{basis}. Observing that 
\[
\frac{1}{1+(p-2)q}+\frac{(p-2)}{\frac1q+(p-2)}=1,
\]
it follows from \eqref{algebra}
along with the help of the Young inequality that (using also the
fact that $q_1\le q$)
$$
\begin{aligned}
\abs{A(t,x,\nabla v)-A(t,x,\nabla u)}^{p'q_1}&\le  c\abs{\nabla v}^\frac{(p-2)p q_1}{p-1}\abs{\nabla g}^\frac{pq_1}{p-1}+c|\nabla g|^{pq_1} +c\\
 &= c\abs{\nabla v}^\frac{(p-2)p }{\frac1q+p-2}\abs{\nabla g}^\frac{pq}{1+q(p-2)}+c|\nabla g|^{pq} +c \\
 &\leq c\left(1+\abs{\nabla v}^p+\abs{\nabla g}^{pq}\right).
\end{aligned}
$$
Applying this inequality to
\eqref{basis} with $\tilde{q}:=q_1$ and recalling that $p'\le p$
(since $p\ge 2$),
we find that
\begin{equation}
\begin{split}
&\int_0^T\int_\Omega \abs{\nabla v}^{pq_1} \dx \dt
\\
&\quad \leq c\Bigg(1+  \int_0^T\int_\Omega \abs{\nabla v}^{p} \dx \dt  +\norm{g}_{X^{p,q}}^{pq_1} + \norm{f}^{p'q_1}_{L^{p'q}(0,T;L^{p'q}(\Omega; \mathbb{R}^{n\times N}))} \\
& \qquad \qquad \qquad \qquad \qquad \qquad + \norm{g}_{X^{p,1}}^{p+\frac{p^2(q_1-1)}{2}} + \norm{f}_{L^{p'}(0,T;L^{p'}(\Omega; \mathbb{R}^{n\times N}))}^{p'+\frac{p'p(q_1-1)}{2}} \Bigg)<\infty.
\end{split}\label{basis1}
\end{equation}
Because $|\nabla v|^p$ is integrable
over $\Omega_T$, we have improved the integrability of $\nabla v$ to
$L^{p q_1}$. Returning to
\eqref{basisq}, we see that the right hand side finite and deduce
the a~prior estimate \eqref{basisqq}.
We next continue this process
iteratively for
$$
q_k:=\frac{(p-1)q_{k-1}}{\frac{q_{k-1}}{q}+p-2}
$$
such that $q_k<q_{k+1}<q$ and $q_k \to q$ as $k\to \infty$. We write
the corresponding dual exponents as
\[
\frac{1}{\big(\frac{q_{k-1}}{q}+(p-2)\big)\frac{q}{q_{k-1}}}+\frac{(p-2)}{\frac{q_{k-1}}q+(p-2)}=1,
\]
to see that
 \[
 \abs{\nabla v}^\frac{(p-2)p q_k}{p-1}\abs{\nabla g}^\frac{pq_k}{p-1}
 = \abs{\nabla v}^\frac{(p-2)pq_{k-1} }{\frac{q_{k-1}}{q}+p-2}\abs{\nabla g}^\frac{p q_{k-1}}{\frac{q_{k-1}}{q}+p-2}\leq \abs{\nabla v}^{pq_{k-1}}+\abs{\nabla g}^{pq}.
\]
We then combine this inequality with \eqref{algebra} to observe that
(using $q_k\le q$ again)
\begin{equation*}
\begin{split}
\abs{A(t,x,\nabla v)-A(t,x,\nabla u)}^{p'q_k}
&\le c(\abs{\nabla v}^{(p-2)p'q_k}\abs{\nabla g}^{p'q_k} +\abs{\nabla g}^{p q_k}+1)\\
&\le c\left(1+\abs{\nabla v}^{pq_{k-1}}+\abs{\nabla g}^{pq}\right).
\end{split}
\end{equation*}
Since we have assumed that $\abs{\nabla v}\in L^{pq_{k-1}}$, we can
now use the above inequality in \eqref{basis} to reach improved
integrability result $\abs{\nabla v} \in L^{pq_k}$. Consequently, we
can now use \eqref{basisq} with $\tilde{q}:=q_k$ that has now the
finite right hand side, to further get the uniform estimate
\eqref{basisqq} with $\tilde{q}:=q_k$
\begin{equation}
\begin{split}
\int_0^T\int_\Omega \abs{\nabla v}^{p q_k} \dx \dt &\leq c\Bigg(1 +\norm{g}_{X^{p,q}}^{pq_k} + \norm{f}^{p'q_k}_{L^{p'q}(0,T;L^{p'q}(\Omega; \mathbb{R}^{n\times N}))} \\
& \qquad \qquad \qquad + \norm{g}_{X^{p,1}}^{p+\frac{p^2(q_k-1)}{2}} + \norm{f}_{L^{p'}(0,T;L^{p'}(\Omega; \mathbb{R}^{n\times N}))}^{p'+\frac{p'p(q_k-1)}{2}} \Bigg).
\end{split}\label{basisqqq}
\end{equation}
We let $k\to \infty$ in \eqref{basisqqq} to deduce
\eqref{goal}, as required. The proof is complete.


\begin{thebibliography}{10}
\bibitem{AceM05}
E.~Acerbi and G.~Mingione, \emph{Gradient estimates for the
  {$p(x)$}-{L}aplacean system}, J. Reine Angew. Math. \textbf{584} (2005),
  117--148.

\bibitem{AceMin07}
\bysame, \emph{Gradient estimates for a class of parabolic systems}, Duke Math.
  J. \textbf{136} (2007), no.~2, 285--320.

\bibitem{Bog14}
V.~B\"ogelein, \emph{Global {C}alder\'on-{Z}ygmund theory for nonlinear
  parabolic systems}, Calc. Var. Partial Differential Equations \textbf{51}
  (2014), no.~3-4, 555--596.

\bibitem{Bog15}
V.~B\"{o}gelein, \emph{Global gradient bounds for the parabolic p-laplacian
  system}, Proceedings of the London Mathematical Society \textbf{111} (2015),
  no.~3, 633--680.

\bibitem{BCDKS17}
D.~Breit, A.~Cianchi, L.~Diening, T.~Kuusi, and S.~Schwarzacher, \emph{The
  {$p$}-{L}aplace system with right-hand side in divergence form: inner and up
  to the boundary pointwise estimates}, Nonlinear Anal. \textbf{153} (2017),
  200--212. \MR{3614668}

\bibitem{BreDieSch13}
D.~Breit, L.~Diening, and S.~Schwarzacher, \emph{Solenoidal {L}ipschitz
  truncation for parabolic {PDE}s}, Math. Models Methods Appl. Sci. \textbf{53}
  (2013), no.~14, 2671--2700.

\bibitem{BreStrVer18}
D.~Breit, B.~Stroffolini, and A.~Verde, \emph{Non-stationary flows of
  asymptotically {N}ewtonian fluids}, Commun. Contemp. Math. \textbf{20}
  (2018), no.~2, 1750006, 16.

\bibitem{BulDieSch16}
M.~Bul\'\i\v{c}ek, L.~Diening, and S.~Schwarzacher, \emph{Existence, uniqueness
  and optimal regularity results for very weak solutions to nonlinear elliptic
  systems}, Anal. PDE \textbf{9} (2016), 1115--1151.

\bibitem{BulSch16}
M.~Bul\'\i\v{c}ek and S.~Schwarzacher, \emph{Existence of very weak solutions
  to elliptic systems of {$p$}-{L}aplacian type}, Calc. Var. Partial
  Differential Equations \textbf{55} (2016).

\bibitem{ByunOh18}
S.-S. Byun and J.~Oh, \emph{Global {M}orrey regularity for asymptotically
  regular elliptic equations}, Appl. Math. Lett. \textbf{76} (2018), 227--235.

\bibitem{ByunOhWang15}
S.-S. Byun, J.~Oh, and L.~Wang, \emph{Global {C}alder\'on-{Z}ygmund theory for
  asymptotically regular nonlinear elliptic and parabolic equations}, Int.
  Math. Res. Not. IMRN (2015), no.~17, 8289--8308.

\bibitem{ByunOk16SIAM}
S.-S. Byun and J.~Ok, \emph{Nonlinear parabolic equations with variable
  exponent growth in nonsmooth domains}, SIAM J. Math. Anal. \textbf{48}
  (2016), no.~5, 3148--3190.

\bibitem{ByunOk16}
\bysame, \emph{On {$W^{1,q(\cdot)}$}-estimates for elliptic equations of
  $p(x)$-{L}aplacian type}, J. Math. Pures Appl. (9) \textbf{106} (2016),
  no.~3, 512--545.

\bibitem{ByunOkRyu13}
S.-S. Byun, J.~Ok, and S.~Ryu, \emph{Global gradient estimates for general
  nonlinear parabolic equations in nonsmooth domains}, J. Differential
  Equations \textbf{254} (2013), no.~11, 4290--4326.

\bibitem{ByunPala14}
S.-S. Byun and D.K. Palagachev, \emph{Morrey regularity of solutions to
  quasilinear elliptic equations over {R}eifenberg flat domains}, Calc. Var.
  Partial Differential Equations \textbf{49} (2014), no.~1-2, 37--76.

\bibitem{ByunWang07}
S.-S. Byun and L.~Wang, \emph{Parabolic equations in time dependent
  {R}eifenberg domains}, Adv. Math. \textbf{212} (2007), no.~2, 797--818.

\bibitem{ByunWangZhou07}
S.-S. Byun, L.~Wang, and S.~Zhou, \emph{Nonlinear elliptic equations with {BMO}
  coefficients in {R}eifenberg domains}, J. Funct. Anal. \textbf{250} (2007),
  no.~1, 167--196.

\bibitem{CafPer98}
L.~A. Caffarelli and I.~Peral, \emph{On {$W^{1,p}$} estimates for elliptic
  equations in divergence form}, Comm. Pure Appl. Math. \textbf{51} (1998),
  no.~1, 1--21.

\bibitem{CM1}
M.~Colombo and G.~Mingione, \emph{Calderón-zygmund estimates and non-uniformly
  elliptic operators}, J. Funct. Anal. \textbf{270} (2007), no.~4, 1416--1478.

\bibitem{DiB93}
E.~DiBenedetto, \emph{Degenerate parabolic equations}, Springer-Verlag, New
  York, 1993.

\bibitem{DiBFri85}
E.~DiBenedetto and A.~Friedman, \emph{H\"older estimates for nonlinear
  degenerate parabolic systems}, J. Reine Angew. Math. \textbf{357} (1985),
  1--22.

\bibitem{DieRuzWol10}
L.~Diening, M.~\Ruzicka, and J.~Wolf, \emph{Existence of
  weak solutions for unsteady motions of generalized {N}ewtonian fluids}, Ann.
  Sc. Norm. Super. Pisa Cl. Sci. (5) \textbf{9} (2010), no.~1, 1--46.

\bibitem{Evans10}
L.C. Evans, \emph{Partial differential equations}, second ed., Graduate Studies
  in Mathematics, vol.~19, American Mathematical Society, Providence, RI, 2010.

\bibitem{Iwa82}
T.~Iwaniec, \emph{On ${L^p}$-integrability in pde's and quasiregular mappings
  for large exponents}, Annales Academi{\ae} Scientiarum Fennic{\ae}, Series A.
  I. Mathematica \textbf{7} (1982), 301--322.

\bibitem{Iwa83}
T.~Iwaniec, \emph{Projections onto gradient fields and {$L^{p}$}-estimates for
  degenerated elliptic operators}, Studia Math. \textbf{75} (1983), no.~3,
  293--312.

\bibitem{KinZho01}
J.~Kinnunen and S.~Zhou, \emph{A boundary estimate for nonlinear equations with
  discontinuous coefficients}, Differential Integral Equations \textbf{14}
  (2001), no.~4, 475--492.

\bibitem{KuuMin12}
T.~Kuusi and G.~Mingione, \emph{New perturbation methods for nonlinear
  parabolic problems}, J. Math. Pures Appl. (9) \textbf{98} (2012), no.~4,
  390--427.

\bibitem{Lie88}
G.M. Lieberman, \emph{Boundary regularity for solutions of degenerate elliptic
  equations}, Nonlinear Anal. \textbf{12} (1988), no.~11, 1203--1219.

\bibitem{Lions69}
J.-L. Lions, \emph{Quelques m\'ethodes de r\'esolution des probl\`emes aux
  limites non lin\'eaires}, Dunod; Gauthier-Villars, Paris, 1969.

\bibitem{MengPhuc12}
T.~Mengesha and N.C. Phuc, \emph{Global estimates for quasilinear elliptic
  equations on {R}eifenberg flat domains}, Arch. Ration. Mech. Anal.
  \textbf{203} (2012), no.~1, 189--216.

\bibitem{Mi1}
Giuseppe Mingione, \emph{The calderón-zygmund theory for elliptic problems
  with measure data}, Ann. Sc. Norm. Super. Pisa Cl. Sci. (5) \textbf{6}
  (2007), no.~2, 195--261.

\bibitem{Minty63}
G.~J. Minty, \emph{On a ``monotonicity'' method for the solution of non-linear
  equations in {B}anach spaces}, Proc. Nat. Acad. Sci. U.S.A. \textbf{50}
  (1963), 1038--1041.

\bibitem{PalaSoft11}
D.K. Palagachev and L.G. Softova, \emph{The {C}alder\'on-{Z}ygmund property for
  quasilinear divergence form equations over {R}eifenberg flat domains},
  Nonlinear Anal. \textbf{74} (2011), no.~5, 1721--1730.

\bibitem{Ph1}
Nguyen~Cong Phuc, \emph{Nonlinear muckenhoupt-wheeden type bounds on reifenberg
  flat domains, with applications to quasilinear riccati type equations}, Adv.
  Math. \textbf{250} (2014), 387--419.

\bibitem{Show97}
R.E. Showalter, \emph{Monotone operators in {B}anach space and nonlinear
  partial differential equations}, Mathematical Surveys and Monographs,
  vol.~49, American Mathematical Society, Providence, RI, 1997.

\bibitem{Uhl77}
K.~Uhlenbeck, \emph{Regularity for a class of non-linear elliptic systems},
  Acta Math. \textbf{138} (1977), no.~3-4, 219--240.

\bibitem{Ura68}
N.~N. Ural'ceva, \emph{Degenerate quasilinear elliptic systems}, Zap.
  Nau\v cn. Sem. Leningrad. Otdel. Mat. Inst. Steklov. (LOMI) \textbf{7}
  (1968), 184--222.

\bibitem{SvYa02}
V.~\v{S}ver{\'{a}}k and X.~Yan, \emph{Non-{L}ipschitz minimizers of smooth
  uniformly convex functionals}, Proc. Natl. Acad. Sci. USA \textbf{99} (2002),
  no.~24, 15269--15276.

\bibitem{ZhaZhe16}
J.~Zhang and S.~Zheng, \emph{Lorentz estimate for nonlinear parabolic obstacle
  problems with asymptotically regular nonlinearities}, Nonlinear Anal.
  \textbf{134} (2016), 189--203.

\end{thebibliography}
\end{document}